\documentclass{amsart}
\usepackage{a4wide}
\usepackage{amsfonts}
\usepackage{amssymb}
\usepackage{latexsym}
\usepackage{color}
\usepackage{euscript}
\usepackage{mathrsfs} %\mathscr T
\usepackage{dsfont}
\usepackage{bbold}
\usepackage{tikz}
\usepackage[T1]{fontenc}
\usepackage{setspace}
\usepackage{pgf,tikz}
\usetikzlibrary{arrows}
\definecolor{qqqqcc}{rgb}{0.0,0.0,0.8}
\definecolor{qqqqff}{rgb}{0.0,0.0,1.0}
\definecolor{uuuuuu}{rgb}{0.26666666666666666,0.26666666666666666,0.26666666666666666}

\makeatletter
\@namedef{subjclassname@2010}{%
\textup{2010} Mathematics Subject Classification}
\makeatother
\tikzset{>=latex}
\theoremstyle{plain}
\newtheorem{thm}{Theorem}%[section]
\newtheorem{prop}[thm]{Proposition}
\newtheorem{cor}[thm]{Corollary}
\newtheorem{lem}[thm]{Lemma}

\theoremstyle{definition}

\newtheorem{exa}[thm]{{\it Example}}
\theoremstyle{remark}
\newtheorem{rem}[thm]{{\it Remark}}
\newtheorem*{rem*}{{\it Remark}}

\DeclareMathOperator{\Alg}{{\mathbf{Alg}}}
\DeclareMathOperator{\dess}{{\mathsf{Des}}}
\DeclareMathOperator{\dom}{\mathcal{D}}
\DeclareMathOperator{\dzii}{{\mathsf{Chi}}}

\DeclareMathOperator{\jad}{\mathsf{ker}}
\DeclareMathOperator{\koo}{{\mathsf{root}}}

\DeclareMathOperator{\Lat}{{\mathbf{Lat}}}

\DeclareMathOperator{\paa}{{\mathsf{par}}}

\newcommand*{\card}[1]{\mathrm{card}(#1)}

\newcommand*{\dz}[1]{{\EuScript D}(#1)}
\newcommand*{\dzi}[1]{\dzii(#1)}
\newcommand*{\dzin}[2]{\dzii^{\langle#1\rangle}(#2)}

\newcommand*{\lambdab}{{\boldsymbol\lambda}}
\newcommand*{\mub}{{\boldsymbol\mu}}

\newcommand{\gphi}{\varGamma_{\hat\varphi}}
\newcommand{\gphiw}{\varGamma_{\hat\varphi_w}}
\newcommand*{\mphi}{M_{\hat\varphi}}
\newcommand*{\mphiw}{M_{\hat\varphi_w}}

\newcommand*{\pa}[1]{\paa(#1)}
\newcommand*{\pan}[2]{\paa^{#1}(#2)}

\newcommand*{\slam}{S_{\boldsymbol \lambda}}

\newcommand*{\smalloplus}{\raise0pt\hbox{$\scriptscriptstyle \oplus$}}

\newcommand*{\tcal}{{\mathscr T}}
\newcommand*{\gcal}{{\mathscr S}}

\newcommand*{\zbb}{\mathbb{Z}}
\newcommand*{\obr}[1]{\mathcal{R}(#1)}
\newcommand*{\nul}[1]{\mathcal{N}(#1)}

\newcommand*{\pth}{\mathscr{P}}

\newcommand*{\multi}{\mathcal{M}(\lambdab)}
\newcommand*{\multib}{\mathcal{M}(\lambdab)}
\newcommand*{\multio}{\mathscr{M}(\lambdab)}

\newcommand*{\multisk}{\mathcal{M}_0(\lambdab)}

\def\pp{\EuScript{P}}
\newcommand*{\alphab}{\mathbf{\alpha}}
\newcommand{\Le}{\leqslant}
\newcommand{\Ge}{\geqslant}

\def\is#1#2{\left\langle#1,#2\right\rangle}

\DeclareMathOperator{\D}{d\!}

\def\Dbb{\mathbb D}
\def\N{\mathbb N}
\def\Z{\mathbb{Z}}
\def\R{\mathbb{R}}
\def\C{\mathbb{C}}
\def\T{\mathbb{T}}

\def\ee{\EuScript E}

\def\hh{\EuScript H}
\def\kk{\EuScript K}
\def\mm{\EuScript M}

\def\bsb{{\mathbf B}}

\def\bsf{{\mathbf    F}}
\def\bst{{\mathbf T}}

\DeclareMathAlphabet{\mathpzc}{OT1}{pzc}{m}{it}
\title[Weighted shifts, multiplier algebras, reflexivity and decompositions]{Weighted shifts on directed trees. Their multiplier algebras, reflexivity and decompositions}
\author{P.\ Budzy{\'n}ski}
\author{P.\ Dymek}
\author{A.\ P{\l}aneta}
\author{M.\ Ptak}
\address{Katedra Zastosowa\'n Matematyki, Uniwersytet Rolniczy w Krakowie, ul. Balicka 253c, 30-198 Krak\'ow, Poland}
    \email{piotr.budzynski@ur.krakow.pl}
    \email{piotr.dymek@ur.krakow.pl}
    \email{artur.planeta@.ur.krakow.pl}
    \email{rmptak@cyf-kr.edu.pl}
   
\begin{document}
\setstretch{1.2}

\begin{abstract}
We study bounded weighted shifts on directed trees. We show that the set of multiplication operators associated with an injective weighted shift on a rooted directed tree coincides with the WOT/SOT closure of the set of polynomials of the weighted shift. From this fact we deduce reflexivity of those weighted shifts on rooted directed trees whose all path-induced spectral-like radii are positive. We show that weighted shifts with positive weights on rooted directed trees admit a Wold-type decomposition. We prove that the pairwise orthogonality of the factors in the decomposition is equivalent to the weighted shift being balanced. 
\end{abstract}
\maketitle
\section{Introduction}
We study bounded weighted shifts on directed trees using analytic function theory approach, which was initiated in \cite{b-d-p-2015}. The class of weighted shifts on directed trees contains all classical weighed shifts (see \cite{j-j-s-2012-mams}) and it is related to that of weighted composition operators in $L^2$-spaces (see \cite{b-j-j-s-wco}). It was a source of interesting examples, problems and results (see e.g., \cite{b-d-j-s-2014-aaa, b-j-j-s-2014-jmaa, j-j-s-2012-jfa,j-j-s-2014-pams,p-2016-jmaa,t-2015-jmaa}). Our motivation for the study comes from a variety of results on the unilateral shift (see the monograph \cite{nik}) and on classical weighted shifts (see \cite{shi}) that were obtained with help of analytic functions. These elucidate the interplay between theories of analytic functions and operators.

An essential ingredient of our approach to the study of weighted shifts on directed trees is the notion of a multiplier algebra associated to a weighted shift on a directed tree (see \cite{b-d-p-2015}), which consists of coefficients of analytic functions. We prove that the set of multiplication operators with symbols belonging to the multiplier algebra corresponding to an injective weighted shift on a rooted directed tree is equal to the closure of polynomials of the shift in weak/strong operator topology (see Theorem \ref{wielomian}). Building on this, we prove that injective weighted shifts on rooted directed trees which behave well along the paths of the tree are reflexive (see Theorem \ref{refPB}). Both the results are well-known in the context of classical weighted shifts (see \cite{shi}. Recall that the reflexivity of the classical (unweighted) unilateral shift was proved in the  \cite{sa}; in turn, the reflexivity of some non--injective weighted shifts was shown in \cite{almp,kpmp}. Later in the paper, we solve two problems concerning multiplier algebras that were asked in \cite{b-d-p-2015} (see Examples \ref{mad} and \ref{offspring}).
In the final section, we turn our attention to the possibility of decomposing a weighted shift on a directed tree into the orthogonal sum of the restrictions of all the powers of the shift to the kernel of its adjoint. This can be done in case of a balanced injective weighted shift on a rooted directed tree (see Theorem \ref{studenciPB}). Without the assumption of the shift being balanced we get a weaker Wold-type decomposition (see Theorem \ref{studenciPB}). 

It is worth noting that the analytic aspects of the theory of weighted shifts on directed trees were studied also in \cite{c-t} and \cite{c-p-t}. The approach used therein was different than ours and relied on Shimurin's work (see \cite{s-2001-jram}) employing vector-valued analytic functions.

\section{Preliminaries}
Let $\N$, $\R$ and $\C$ denote the set of all natural numbers, real numbers and complex numbers, respectively. Set $\N_0=\N\cup\{0\}$ and $\R_+=[0,\infty)$. Denote by $\T$ the unit circle $\{z\in\C\colon |z|=1\}$ and by $\Dbb$ the open unit disc $\{z\in\C\colon |z|<1\}$. If $V$ is a set and $W\subseteq V$, then we write $W^c$ for $V\setminus W$. The characteristic function of $W$ is denoted by $\chi_{_{W}}$. For a set $Y$, $\card{Y}$ denotes the cardinal number of $Y$. Symbol $\C[X]$ stands for the set of all complex polynomials in one variable, whereas $\mathcal{T}$ denotes the set of trigonometric polynomials on $\T$.  In all what follows we use the convention that $\sum_{i\in \emptyset}x_i=0$.

Let $V$ be a nonempty set. Then $\ell^2(V)$ denotes the Hilbert space of all functions $f\colon V\to\C$ such that $\sum_{v\in V}|f(v)|^2<\infty$ with the inner product given by $\is{f}{g}=\sum_{v\in V} f(v)\overline{g(v)}$ for $f,g\in \ell^2(V)$. The norm induced by $\is{\cdot}{-}$ is denoted by $\| \cdot \|$. For $u \in V$, we define $e_u \in \ell^2(V)$ to be the characteristic function of the one-point set $\{u\}$; clearly, $\{e_u\}_{u\in V}$ is an orthogonal basis of $\ell^2(V)$. We will denote by $\ee$ the linear span of $\{e_u\}_{u\in V}$. Given a subset $W$ of $V$, $\ell^2(W)$ stands for the subspace of $\ell^2(V)$ composed of all functions $f$ such that $f(v)=0$ for all $W^c$, and $\ee_{W}$ denotes the set of all $f\in \ell^2(W)$ such that $\{v\in V\colon f(v)\neq 0\}$ is finite. By $P_{\kk}$ we denote the orthogonal projection from $\ell^2(V)$ onto its closed subspace $\kk$.

Let $\hh$ be a Hilbert space, $J$ be a nonempty set, and $\{X_j\}_{j\in J}\subseteq \hh$ be a family of sets. Then $\bigvee_{j \in J} X_j$ stands for the smallest closed linear subspace of $\hh$ such that $X_i \subset \bigvee_{j \in J} X_j$ for every $i \in J$. Throughout the paper, unless otherwise stated, every subspace of a Hilbert space is assumed to be closed.

Let $\hh$ be a (complex) Hilbert space. If $A$ is a (linear) operator in $\hh$, then $\dom(A)$, $\nul{A}$, $\obr{A},$ and $A^*$ denote the domain, the kernel, the range, and the adjoint of $A$, respectively (in case it exists). We write $\bsb(\hh)$ for the algebra of all bounded operators on $\hh$ equipped with the standard operator norm. By $\bsf_1(\hh)$ and $\bst(\hh)$ we denote the sets of rank one and trace class, respectively, operators on $\hh$. Let $\mathscr{W}$ be a subalgebra of $\bsb(\hh)$. Then the preannihilator  $\mathscr{W}_\perp$ of $\mathscr{W}$ is given by $\{T\in\bst(\hh)\colon \mathrm{tr}(AT)=0\text{ for all }A\in \mathscr{W}\}$. The set of all invariant subspaces of all operators $A\in\mathscr{W}$ is denoted by $\Lat \mathcal{W}$; recall that a (closed) subspace $\kk$ of $\hh$ is invariant for $A\in\bsb(\hh)$ if $A\kk\subset \kk$. Given a set $\mathscr{V}\subseteq\bst(\hh)$ we set $\mathscr{V}^\perp=\{A\in\bsb(\hh)\colon \mathrm{tr}(AT)=0\text{ for all } T\in \mathscr{V}\}$. If $\mathcal{M}$ is a family of subspaces of $\hh$, then we set $\Alg \mathcal{M} = \{ A \in \bsb(\hh)\colon A\mm\subset \mm \text{ for every } \mm \in \mathcal{M}\}$. The algebra $\mathscr{W}$ is said to be {\em reflexive} if $\Alg \Lat \mathscr{W} = \mathscr{W}$. Given $A\in\bsb(\hh)$, $\mathcal{W}(A)$ denotes the smallest algebra containing $A$ and the identity operator $I$ and closed in the weak operator topology; if $\mathscr{W}(A)$ is reflexive, then $A$ is said to be {\em reflexive}. Note that $\Lat A=\Lat \mathscr{W}(A)$.

Let $\tcal=(V,E)$ be a directed tree ($V$ and $E$ stand for the sets of vertices and directed edges of $\tcal$, respectively). Set $\dzi u= \{v\in V\colon (u,v)\in E\}$ for $u \in V$. Denote by $\paa$ the partial function from $V$ to $V$ which assigns to a vertex $u\in V$ its parent $\pa{u}$ (i.e.\ a unique $v \in V$ such that $(v,u)\in E$). For $k\in \N$, $\paa^k$ denotes the $k$-fold composition of the partial function $\paa$; $\paa^0$ denotes the identity map on $V$. A vertex $u \in V$ is called a {\em root} of $\tcal$ if $u$ has no parent. A root is unique (provided it exists); we denote it by $\koo$. The tree $\tcal$ is {\em rooted} if the root exists. The tree $\tcal$ is {\em leafless} if $\card{\dzi{v}}\Ge 1$ for every $v\in V$. 
Suppose $\tcal$ is rooted. We set $V^\circ=V\setminus \{\koo\}$. If $v\in V$, then $|v|$ denotes the unique $k\in\N_0$ such that $\paa^k(v)=\koo$. Given $n\in\N_0$, $\{|v|=n\}$ stands for the set $\{v\in V\colon |v|=n\}$. For given $u\in V$ and $n\in\N_0$ we set $\dess(u)=\{ v\in  V\colon \pan{k}{v}=u \text{ for some } k\in\N_0\}$ and $\dzin{n}{u}=\{v\in V\colon \pan{n}{v}=u\}$. A subgraph $\gcal$ of $\tcal$ which is a directed tree itself is called a {\em subtree} of $\tcal$. A {\em path} in $\tcal$ is a subtree $\pth=(V_\pth,E_\pth)$ of $\tcal$ which satisfies the following two conditions: (i) $\koo\in\pth$, (ii) for every $v\in V_\pth$, $\card{\mathsf{Chi}_{\pth}(v)}=1$. The collection of all paths in $\tcal$ is denoted by $\pp=\pp(\tcal)$. We refer the reader to \cite{j-j-s-2012-mams} for more on directed trees.

Weighted shifts on directed trees are defined as follows. Let $\tcal=(V,E)$ be a directed tree and let $\lambdab=\{\lambda_v\}_{v \in V^{\circ}} \subseteq \C$ be such that 
\begin{align*}
\sup_{v\in V} \sum_{u\in\dzii(v)}|\lambda_v|^2<\infty.
\end{align*} 
Then the following formula
\begin{align*}
(\slam f) (v) =
   \begin{cases}
\lambda_v \cdot f\big(\pa v\big) & \text{ if } v\in V^\circ,
   \\
0 & \text{ if } v=\koo,
   \end{cases}\quad f\in\ell^2(V),
\end{align*}
defines a bounded operator $\slam$ on $\ell^2(V)$ (see \cite[Proposition 3.1.8]{j-j-s-2012-mams}), which is called the {\em weighted shift on $\tcal$ with weights} $\lambdab$.  The reader is referred to \cite{j-j-s-2012-mams} for the foundations of the theory of weighted shifts on directed trees.

To avoid unnecessary repetitions we gather below a few  basic assumptions that will be used throughout the paper:
\begin{align} \label{stand1}\tag{$\dag$}
   \begin{minipage}{65ex}
$\tcal=(V,E)$ is a countably infinite rooted directed tree, and $\lambdab=\{\lambda_v\}_{v \in V^\circ} \subseteq  (0,\infty)$.
   \end{minipage}
   \end{align}
Recall that any weighted shift on a directed tree with non-zero weights is unitarily equivalent to a weighted shift with positive weights (see \cite[Theorem 3.2.1]{j-j-s-2012-mams}).

In our previous work we used a notion of a multiplier algebra induced by a weighted shift, which is defined via related multiplication operators. These are given as follows. Assume that
\begin{align} \label{stand2}\tag{$\star$}
   \begin{minipage}{65ex}
$\tcal=(V,E)$ is a countably infinite rooted and leafless directed tree,  and $\lambdab=\{\lambda_v\}_{v \in V^\circ}\subseteq  (0,\infty)$.
   \end{minipage}
   \end{align}
Given $u\in V$ and $v\in\dess(u)$ we set 
\begin{align*}
\lambda_{u|v}=\begin{cases}
 1 & \text{ if } u=v,\\
 \prod_{n=0}^{k-1} \lambda_{\pan{n}{v}} & \text{ if } \pan{k}{v}=u.\end{cases}    
\end{align*}   
Now, let $\hat\varphi \colon \N_0 \to \C$. Define the mapping $\varGamma_{\hat\varphi}^\lambdab\colon \C^V\to \C^V$ by
\begin{align}\label{multi1}
\big(\varGamma_{\hat\varphi}^\lambdab f\big)(v) =\sum_{k=0}^{|v|} \lambda_{\paa^k(v)|v} \, \hat\varphi (k) f(\paa^k(v)),\quad v\in V.
\end{align}
The {\em multiplication operator} $M_{\hat\varphi}^{\lambdab}\colon \ell^2(V)\supseteq \dz{M_{\hat\varphi}^{\lambdab}}\to \ell^2(V)$  is given by
   \begin{align}\label{multi2}
   \begin{aligned}
\dz{M_{\hat\varphi}^{\lambdab}} & = \big\{f \in \ell^2(V) \colon \varGamma_{\hat\varphi}^\lambdab f \in \ell^2(V)\big\},
   \\
M_{\hat\varphi}^{\lambdab} f & = \varGamma_{\hat\varphi}^\lambdab f, \quad f \in \dz{M_{\hat\varphi}^{\lambdab}}.
\end{aligned}
\end{align}
We call $\hat\varphi \colon \N_0\to \C$ the {\em symbol} of $M_{\hat\varphi}^{\lambdab}$. If no confusion can arise, we write $\varGamma_{\hat\varphi}$ and $M_{\hat\varphi}$ instead of $\varGamma_{\hat\varphi}^\lambdab$ and $M_{\hat\varphi}^{\lambdab}$, respectively. As shown in \cite[Lemma 4.1]{b-d-p-2015}, any multiplication operator $\mphi$ is automatically closed. Thus, if $\dz{\mphi}=\ell^2(V)$, then $\mphi\in\bsb(\ell^2(V))$. It is easily seen that for $u\in V$ such that $e_u\in\dz{\mphi}$ we have
\begin{align}\label{dziobak}
(\mphi e_u)(v)=\sum_{k=0}^\infty \hat\varphi(k)\slam^k e_u(v)=\left\{\begin{array}{cl}
\lambda_{u|v} \hat\varphi(n) & \text{ if } v\in\dzin{n}{u},\ n\in\N_0\\
0 & \text{ otherwise}.
\end{array}\right.,\quad v\in V
\end{align}
and
\begin{align}\label{Dziobak}
\mphi e_u=\sum_{v\in\dess(u)} \lambda_{u|v}\hat{\varphi}(|v|-|u|)e_v.
\end{align}
By $\multi$ we denote the {\em multplier algebra induced by $\slam$}, i.e., the commutative Banach algebra consisting of  all $\hat\varphi\colon \N_0\to \C$ such that $\dz{\mphi}= \ell^2(V)$ with the norm
\begin{align*}
\|\hat\varphi\|:=\|\mphi\|,\quad \hat\varphi\in \multi.
\end{align*}
Throughout the paper $\multisk$ stands for the set of all functions from $\multi$ having finite supports. The set $\big\{M_{\hat\varphi}\colon \hat\varphi\in\multi\big\}$ will be denoted by $\multio$. I is worth noting that if $\|\slam\| = r(\slam)$, then the algebra $\multi$ consists of all sequences that come as coefficients of bounded analytic functions on the disc $\{z \in \C \colon |z|< \|\slam\|\}$ (this follows from \cite[Propositions 4.6 \& 4.7]{b-d-p-2015}; cf. \cite[Corollary, p. 76]{shi}). For more information on $\multi$ we refer the reader to \cite{b-d-p-2015}.

%Let $\pth = (V_\pth,E_\pth) \in \pp$. Note, that the set $\multip$ is composed of multiplier functions with domain $\ell^2(V_\pth)$. If $\hat \varphi \in \multip$, then by $M^\pth_{\hat\varphi}$ we denote the multiplier operator defined above with domain $\ell^2(V_\pth)$ and generated by $\hat \varphi$. Note that if $\hat \varphi \in \multib$ then $\hat \varphi \in \multip$.

The following well-known lemma will be used later in the paper. We use the notation: given a subspace $\kk$ of a Hilbert space $\hh$, $P_\kk$ denotes the orthogonal projection from $\hh$ onto $\kk$, while $\kk^\perp$ stands for the orthogonal complement of $\kk$ in $\hh$.
\begin{lem} \label{prosty}
Let $\hh$ be a Hilbert space and let $\mm$ and $\kk$ be its closed linear subspaces. Then $P_{\kk}$ commutes with $P_{\mm}$ if and only if $\kk = (\kk \cap \mm) \oplus \big(K \cap (\mm^\perp)\big)$, or equivalently,  
if and only if $P_\kk P_\mm \hh\subseteq P_\mm \hh$. 
\end{lem}

\section{Polynomial approximation}
In the following section we prove that the set of multiplication operators associated with a weighted shift on a directed tree lies in the closure of polynomials of the weighted shift in the topologies of strong and weak operator convergence. We follow here the idea of Shields (see \cite{shi}) who proved a corresponding result for classical weighted shifts.

For $w \in \T$, $f\colon V\to\C$, and $\hat{\varphi} \colon \N_0 \to \C$ we define $f_w\colon V\to\C$ and $\hat \varphi_w\colon\N_0\to \C$ by
\begin{align*}
f_w(u) = w^{|u|} f(u),\quad u\in V,
\end{align*}
and
\begin{align*}
\hat \varphi_w(n) = w^n \hat \varphi(n),\quad n\in\N_0.
\end{align*}
The assignments $w\mapsto f_w$ and $w\mapsto \hat\varphi_w$ have the following properties. 
\begin{lem}\label{piatek}
Assume that \eqref{stand2} holds and $\slam\in\bsb\big(\ell^2(V)\big)$. Then the following assertions are satisfied$:$
\begin{enumerate}
\item[(i)] For every $w \in \T$ and $f \in \ell^2(V)$ the function $f_w$ belongs to $\ell^2(V)$ and $\|f\| = \|f_w\|$.
\item[(ii)] For every $w \in \T$ and  $\hat\varphi \in \multi$ the function $\hat \varphi_w$ belongs to $\multib$, $\|\hat \varphi\| =\|\hat \varphi_w\|$, and
\begin{align}\label{dzem}
M_{\hat\varphi_w} f =\big( M_{\hat\varphi} f_{\bar w}\big)_w,\quad f\in\ell^2(V).
\end{align}
\item[(iii)] For every $f\in\ell^2(V)$ the mapping  $\T \ni w \to f_w \in \ell^2(V)$ is continuous.
\item[(iv)] For every $\hat \varphi \in \multib$ the function $\T \ni w \to M_{\hat \varphi_w} \in \bsb(\ell^2(V))$
is continuous in the strong operator topology.
\end{enumerate}
\begin{proof}
(i) This follows from the equality $|w|=1$.

(ii) Let $w\in \T$. By \eqref{multi1} we have
\begin{align}\notag
\big(\gphiw f\big)(v)&=\sum_{k=0}^{|v|} \lambda_{\paa^k{(v)}|v} w^k \hat \varphi(k) f (\paa^k(v))\\\notag
&=\sum_{k=0}^{|v|} \lambda_{\paa^k{(v)}|v}\hat \varphi(k) w^{|v|} \overline{w}^{|\paa^k(v)|}f(\paa^k(v)) \\\notag 
&=w^{|v|}\sum_{k=0}^{|v|} \lambda_{\paa^k{(v)}|v}\hat \varphi(k) 
\overline{w}^{|\paa^k(v)|}f(\paa^k(v))\\\label{falsyfikat}
&=\big(\gphi f_{\overline{w}}\big)_w(v), \quad v\in V,\ f\in \C^V.
\end{align}
This, according to (i) and \eqref{multi2}, implies
\begin{align}\label{dzmfw}
\big\{f\in\ell^2(V)\colon f_{\bar w}\in \dz{M_{\hat\varphi}}\big\}=\dz{\mphiw}.
\end{align}
Since $\hat\varphi\in\multib$, $\dz{\mphi}=\ell^2(V)$ and thus by \eqref{dzmfw} we get $\dz{\mphiw}=\ell^2(V)$. Hence $\hat\varphi_w\in\multib$. Clearly, \eqref{falsyfikat} yields \eqref{dzem}, which, combined with (i), implies the equality $\|\mphiw f\|=\|\mphi f\|$ for every $f\in\ell^2(V)$. Thus we get $\|\hat\varphi\|=\|\hat\varphi_w\|$.

(iii) Let $f\in \ell^2(V)$ and $w_0 \in \T$. Let $\varepsilon >0$. There exists a finite set $W \subseteq V$ such that $\|\chi_{_{W^c}} f\|=\sum_{u \in W^c} |f(u)|^2 < \varepsilon$. Then, by (i),  $\|\chi_{_{W^c}} f_w\|^2=\|(\chi_{_{W^c}} f)_w\|^2 < \varepsilon$ for every $w \in \T$.
For $w$ close enough to $w_0$ we have 
\begin{align*}
\|\chi_{_W} (f_w-f_{w_0})\|^2=\sum_{u \in W} |w^{|u|} - w_0^{|u|}|^2 |f(u)|^2 < \varepsilon,
\end{align*}
since the set $W$ is finite. Consequently, we get 
\begin{align*}
\|f_w - f_{w_0}\|\Le \|(f_w - f_{w_0}) \chi_{_W}\|+\|\chi_{_{W^c}}f_w\|+\|\chi_{_{W^c}}f_{w_0}\| < 3 \varepsilon^{1/2}.
\end{align*}
This gives (iii).

(iv) Let $f\in\ell^2(V)$, $w_0 \in \T$, and $\varepsilon>0$. Then there exists a finite set $W_1 \subseteq V$ such that $\|\chi_{_{W_1^c}} f \| < \varepsilon$. Hence, by (ii), we get
\begin{align}\label{dolary}
\|M_{\hat \varphi_w} ( \chi_{_{W_1^c}} f)\| \Le \varepsilon\| \hat \varphi \|, \quad w\in\T.
\end{align}
In view of \eqref{Dziobak}, we have 
\begin{align} \label{Wojtek}
\big(M_{\hat \varphi_w} - M_{\hat \varphi_{w_0}}\big) e_u  =  \sum_{ v \in \dess(u)} \Big(w^{|v| -|u|} - w_0^{|v|-|u|} \Big) \lambda_{u|v} \hat \varphi (|v| -|u|) e_v,\quad w\in \T,\ u\in V.
\end{align}
Since $W_1$ is finite, there exists a finite set $W_2 \subseteq V$ such that 
\begin{align}\label{mobot}
\|\chi_{_{W_2^c}}\mphi e_u \|^2=\sum_{v \in \dess(u) \cap W_2^c } \big|\lambda_{u|v} \hat \varphi (|v| -|u|)\big|^2 < \varepsilon^2,\quad u \in W_1.
\end{align}
Hence, by \eqref{Wojtek} and \eqref{mobot}, we get
\begin{align}\label{negatyw}
\Big\|\chi_{_{W_2^c}}\big(M_{\hat \varphi_w} - M_{\hat \varphi_{w_0}}\big) e_u\Big\|^2
&\Le 4 \sum_{v \in \dess(u) \cap W_2^c } \big|\lambda_{u|v} \hat \varphi (|v| -|u|)\big|^2< 4\varepsilon^2,\quad w\in \T,\ u\in W_1.
\end{align}
Since $W_1$, $W_2$ are finite, for every $w$ close enough to $w_0$ the equality $\big|w^{|v|-|u|} - w_0^{|v|-|u|}\big|^2 < \varepsilon$ is satisfied for all $u \in W_1$ and $v \in \dess(u) \cap W_2$. Thus, by \eqref{Wojtek}, we have
\begin{align}\label{negatyw2}
\Big\|\chi_{_{W_2}}\big(M_{\hat \varphi_w} - M_{\hat \varphi_{w_0}}\big) e_u \Big\|^2 \Le \varepsilon \, \|\mphi e_u\|^2 \Le \varepsilon^2 \|\hat \varphi\|^2,\quad u\in W_1.
\end{align}
Combining \eqref{negatyw} and \eqref{negatyw2} we deduce that for every $w$ close enough to $w_0$ we have
\begin{align*}
\Big\|\big(M_{\hat \varphi_w} - M_{\hat \varphi_{w_0}}\big) e_u \Big\|< \varepsilon (2+\|\hat\varphi\|),\quad  u \in W_1.
\end{align*}
Therefore, the above and \eqref{dolary} imply that that for every $w$ close enough to $w_0$ we have
\begin{align*}
\|M_{\hat \varphi_w} f - M_{\hat \varphi_{w_0}} f \| & \Le 2\varepsilon \|\hat\varphi \| +\big\|\big(\mphiw -M_{\hat\varphi_{w_0}}\big)( \chi_{_{W_1}} f)\big\| \\
&\Le 2\varepsilon \|\hat\varphi \|+\sum_{u \in W_1}|f(u)|\big\|\big(M_{\hat \varphi_w} - M_{\hat \varphi_{w_0}}\big) e_u \big\|< C\varepsilon
%& <\varepsilon \Big(2\|\hat\varphi\|+(2+\|\hat\varphi\|)\max\{|f(u)|\colon u\in W_1\}\Big),
\end{align*}
with some positive constant $C$. This completes the proof.
\end{proof}
\end{lem}
In view of Lemma \ref{piatek}(iii), given $\hat\varphi\in\multib$, for any continuous function $q\colon\T\to\C$ and any function $f \in \ell^2(V)$ the mapping
\begin{align*}
\T \ni w \to q(w)M_{\hat \varphi_w} f  \in \ell^2(V)  
\end{align*}
is continuous, which implies that the mapping 
\begin{align*}
\T \ni w \to q(w)M_{\hat \varphi_w}  \in \bsb(\ell^2(V))  
\end{align*}
is SOT-integrable with respect to the normalized Lebesgue measure on $\T$ (cf. \cite[Section 3.1.2]{nik2}). As a consequence we may consider a bounded linear operator on $\ell^2(V)$ given by the integral
\begin{align*}
\int_{\T} q(w)M_{\hat \varphi_w} \D w,
\end{align*}
where 
\begin{align*}%\label{Slabacalka}
\Big(\int_{\T} q(w)M_{\hat \varphi_w} \D w \Big )f = \int_{\T} q(w) M_{\hat \varphi_w} f \D w,\quad f\in\ell^2(V). 
\end{align*}
Clearly, we have 
\begin{align}\label{el1}
\Big\|\int_{\T} q(w)M_{\hat \varphi_w} \D w \Big\| \leq \|q\|_1 \|\hat\varphi\|
\end{align}
with $\|q\|_1$ standing for the $L^1$-norm of $q$, and
\begin{align}\label{slabacalka}
\Big\langle \Big(\int_{\T} q(w)M_{\hat \varphi_w} \D w \Big )f, g \Big\rangle = \int_{\T} q(w)\langle M_{\hat \varphi_w} f,g \rangle \D w,\quad f,g\in\ell^2(V). 
\end{align}

Given  $p\in\mathcal{T}$ of degree $n\in\N_0$, i.e., $p(z)=\sum_{k=-n}^n p_k z^k$ with $\{p_k\}_{k=-n}^n\subseteq \C$, we define $\hat p \colon \N_0 \to \C$ to be the mapping such that \begin{align*}
\hat p(k) = \left\{ \begin{array}{cl} 
p_k & \text{if } k \leq n, \\
0 &  \text{if } k>n.
\end{array}
\right.
\end{align*}
Clearly, $\hat p$ has a finite support and thus it belongs to $\multi$ by \cite[Theorem 4.3\,(ii)]{b-d-p-2015}. 

It seems natural to expect that with a reasonable choice of a function $q\colon \T\to\C$ the operator $\int_\T q(w)M_{\hat \varphi_w} \D w$ is a multiplication operator with appropriate symbol. The following result prove this to be the case. The situation seems to be the most interesting for co-analytic polynomials.
\begin{prop}\label{szpinak}
Assume that \eqref{stand2} holds and $\slam\in\bsb\big(\ell^2(V)\big)$. Let $p\in\mathcal{T}$ and $\hat \varphi \in \multi$. Then $\hat p  \, \hat \varphi \in \multi$ and 
\begin{align*}
\int_{\T} p(\overline{w})  M_{\hat \varphi_w} \D w = M_{\hat p \,  \hat \varphi}. 
\end{align*}
\end{prop}
\begin{proof}
The function $\hat p \, \hat \varphi$ has a finite support and thus, by \cite[Theorem 4.3\,(ii)]{b-d-p-2015}, it belongs to $\multi$. 
By the linearity, it suffices to prove the rest of the claim for $p(w)=w^k$, with $k\in\Z$. 

First we consider $k\in\N$ and we prove that
\begin{align}\label{zbyszek}
\int_{\T} w^k  M_{\hat \varphi_w} \D w 
=0,\quad  k\in\N.
\end{align}
To this end, fix $k\in\N$. Let $u\in V$. Then, by \eqref{slabacalka} and \eqref{Dziobak}, for every $v\in V$ we have\allowdisplaybreaks
\begin{align*}
\Big\langle\Big(\int_{\T} w^kM_{\hat \varphi_w} \D w \Big)\, e_u, e_v \Big\rangle 
&=\int_{\T} w^k  \is{M_{\hat \varphi_w} e_u}{e_v} \D w\\
&= \left\{
\begin{array}{cl} 
\lambda_{u|v}\hat\varphi(n)\displaystyle \int_{\T} w^{n+k}  \D w & \text{ if } v \in \dzin{n}{u},\ n\in\N_0 ,\\
0 & \text{otherwise}.
\end{array} \right.\\
&=0.
\end{align*}
This yields \eqref{zbyszek}.

Now, we consider $k\in\N_0$ and we prove that
\begin{align}\label{zbyszek+}
\int_{\T} w^{-k}  M_{\hat \varphi_w} \D w 
=M_{\chi_{_{\{k\}}} \hat \varphi},\quad k\in\N_0.
\end{align}
For this, fix $k\in\N_0$. Let $u\in V$. Combining \eqref{slabacalka} and \eqref{dziobak} we deduce that for every $v\in V$ the following holds\allowdisplaybreaks
\begin{align*}
\Big\langle\Big(\int_{\T} w^{-k} M_{\hat \varphi_w} \D w \Big)\, e_u, e_v \Big\rangle 
&= \int_{\T} w^{-k}\langle M_{\hat \varphi_w} e_u,e_v \rangle \D w\\
&= \left\{ 
    \begin{array}{cl}
    \lambda_{u|v}\hat\varphi(n)\displaystyle \int_{\T} w^{n-k}  \D w & \text{if }  v \in \dzin{n}{u},\ n\in\N_0 ,\\
     0 & \text{otherwise}.
    \end{array} \right.\\
&= \left\{ 
    \begin{array}{cl}
    \lambda_{u|v}\, \hat\varphi(k) & \text{if } v \in \dzin{k}{u},\\
    0                                  & \text{otherwise}.
    \end{array} \right.\\
&= \is{M_{\chi_{_{\{k\}}} \hat \varphi} e_u}{e_v}.
\end{align*}
Using the fact that for $p(w)=w^k$ we have $\hat p=\chi_{_{\{k\}}}$ completes the proof.
\end{proof}
%By Proposition \ref{bezszpinaku} and Proposition %\ref{szpinak} we get the following
%\begin{cor} \label{zima}
%Assume that \eqref{stand2} holds and $\slam\in\bsb\big(\ell^2(V)\big)$. Let $p(w) = \sum_{k=0}^{n} p_k w^k$ and let $p^- (w) = \sum_{k=1}^{n} p_k \bar w^k$ for $w \in \T$ and some $\{p_k\}_{k=0}^n \subset \C$. If $\hat \varphi \in \multi$ then $\widehat {p}  \, \hat \varphi \in \multi$ and 
%\begin{align*}
%\int_{\T} (p+p^-)(\overline{w})  M_{\hat \varphi_w} \D w = M_{\hat p \,  \hat \varphi}. 
%\end{align*}
%\end{cor}
The framework for polynomial approximation within $\multi$ is set by the following result.
\begin{lem}\label{aproks}
Assume that \eqref{stand2} holds and $\slam\in\bsb\big(\ell^2(V)\big)$. Let $\hat\varphi\in\multi$ and let $\{p_n\}_{n=1}^\infty\subseteq\mathcal{T}$. Assume that the following conditions are satisfied$:$
\begin{enumerate}
\item[(a)] %if $k_n(w)=\sum_{j=0}^n \hat k_{n,j}\, w^j$, $w\in\T$, then 
$\alpha:=\sup\big\{ | \hat p_{n}(k)|\colon n\in\N, k\in\N_0\big\}<\infty$, 
\item[(b)] $\lim_{n\to\infty} \hat p_{n}(k)=1$ for every $k\in\N_0$,
\item[(c)] $\beta:=\sup_{n\in\N}\int_\T |p_n(w)|\D w<\infty$.
\end{enumerate}
Then the following assertions hold$:$
\begin{enumerate}
\item[(i)] for every $n\in\N$, $\hat p_{n}\hat\varphi\in\multi$,
\item[(ii)] for every $n\in\N$, $\|M_{\hat p_n\hat\varphi}\|\Le \beta\, \|M_{\hat\varphi}\|$,
\item[(iii)] $M_{\hat p_n\hat\varphi}\to M_{\hat\varphi}$ in the strong operator topology as $n\to\infty$.
\end{enumerate}
\end{lem}
\begin{proof}
For every $n\in\N$, $p_n$  is a polynomial and thus $\hat p_n\hat\varphi$ has a finite support. Now, it suffices to use \cite[Theorem 4.3\,(ii)]{b-d-p-2015} to get (i).

On the basis of Proposition \ref{szpinak}, Lemma \ref{piatek}\,(ii), and \eqref{el1} we have
\begin{align*}
\|M_{\hat p_n\hat\varphi} f\|
=\Big\|\int_{\T} p_n(\overline{w}) M_{ \hat \varphi_w}  f\D w\Big\|\Le \| M_{\hat \varphi}\|\, \|f\|\, \int_{\T} |p_n(\overline{w})|\D w,\  f\in\ell^2(V), n\in\N.
\end{align*}
This and (c) implies (ii).

For the proof of (iii) we first observe that it is sufficient to show that
\begin{align}\label{mleko}
\lim_{n\to\infty}M_{\hat p_n\hat\varphi}e_u= M_{\hat\varphi}e_u,\quad u\in V.
\end{align}
Indeed, if this is satisfied, then the standard argument using (ii) and the approximation in $\ell^2(V)$ by finite linear combinations of $e_u$'s yields $\lim_{n\to\infty}M_{\hat p_n\hat\varphi}f= M_{\hat\varphi}f$ for every $f\in\ell^2(V)$. 
%
%****************************************************************
%
For the proof of \eqref{mleko} fix $u\in V$. In view of \eqref{dziobak} and the assumptions on coefficients of $p_n$'s we have
\begin{align*}
\lim_{n\to\infty}\big(M_{\hat p_n \hat \varphi}e_u\big) (v)
&=
\begin{cases}
\lim_{n\to\infty}\lambda_{u|v} \hat p_n(j)\hat\varphi(j) & \text{if } v\in \dzin{j}{u}, j\in\N_0,\\
0 & \text{otherwise},
\end{cases}\\
&=
\begin{cases}
\lim_{n\to\infty}\lambda_{u|v}\hat\varphi(j) & \text{if } v\in \dzin{j}{u}, j\in\N_0,\\
0 & \text{otherwise},
\end{cases}\\
&=\big(M_{\hat \varphi} e_u\big)(v).
\end{align*}
Moreover, we see that $\big|\big(M_{\hat p_n \hat \varphi}e_u\big) (v)\big|\leqslant \alpha\  \big|\big(M_{\hat \varphi}e_u\big) (v)\big|$ for every $v\in V$ and every $n\in\N$. Applying the Lebesgue dominated convergence theorem we get \eqref{mleko} and complete the proof.
\end{proof}
Recall that Fejer kernels fullfill the assumptions of Lemma \ref{aproks} (see \cite[p. 17]{hof}).
\begin{cor} \label{cesaro}
Assume that \eqref{stand2} holds and $\slam\in\bsb\big(\ell^2(V)\big)$. Let $\hat\varphi\in\multi$ and let $\{p_n\}_{n=1}^\infty\subseteq\mathcal{T}$ be a Fejer kernels i.e.  $p_n(w) = \sum_{k=-n}^n \big(1 - \tfrac{|k|}{n+1}\big)w^k$, $n\in\N$. Then $M_{\hat p_n\hat\varphi}\to M_{\hat\varphi}$ in the strong operator topology as $n\to\infty$.
\end{cor}
That $\multio$ is norm closed was shown in \cite[Theorem 4.3\,(iv)]{b-d-p-2015}. In fact, essentially the same proof can be used to justify the closedness of $\multio$ in the strong operator topology (we include the version of the proof for completeness).  
\begin{prop}\label{sot}
Assume that \eqref{stand2} holds and $\slam\in\bsb\big(\ell^2(V)\big)$. Then the space $\multio$ is closed in the strong operator topology. 
\end{prop}
\begin{proof} 
Let $A\in\bsb(\ell^2(V))$ belong to the closure of $\multio$ in the strong operator topology and let $\{\hat\varphi_\alpha\}_{\alpha\in A}\subseteq \multi$ be a net such that $\lim_{\alpha}M_{\hat\varphi_\alpha}f=Af$ for every $f\in \ell^2(V)$. We first note that for every $k\in\N_0$, the $\lim_\alpha \hat\varphi_\alpha(k)$ exists. Indeed, fixing $k\in\N_0$ we find $v\in V$ such that $|v|=k$. Then, by \eqref{Dziobak}, for such a $v$ we have
\begin{align*}
\lambda_{\koo|v}\,  \hat\varphi_\alpha(k)=\lambda_{\koo|v}\,  \hat\varphi_\alphab(|v|)%=\lim_\omega \sum_{j=0}^{|v|}\lambda_{\paa^j(v)|v}\hat\varphi_\omega(j) e_{\koo}(\paa^j(v))
=\big(M_{\hat\varphi_\alpha}e_{\koo}\big)(v).
\end{align*}
Since $\lim_{\alpha}\big(M_{\hat\varphi_\alpha}e_{\koo}\big)(v)= (Ae_{\koo})(v)$ and $\lambda_{\koo|v}\neq 0$, the limit $\lim_\alpha \hat\varphi_\alpha(k)$ has to exist. Thus we may define the mapping $\hat\varphi\colon \N_0\to\C$ by
\begin{align*}
\hat\varphi(k)=\lim_\alpha \hat\varphi_\alpha(k),\quad k\in\N_0.
\end{align*}
Since for every $f\in\ell^2(V)$ we have
\begin{align*}
\lim_\alpha \sum_{j=0}^{|v|}\lambda_{\paa^j(v)|v}\hat\varphi_\alpha(j) f\big(\paa^j(v)\big)
=\sum_{j=0}^{|v|}\lambda_{\paa^j(v)|v}\hat\varphi(j) f\big(\paa^j(v)\big),\quad v\in V,
\end{align*}
we deduce that $\big(A f\big)(v)=\big(\varGamma_{\hat\varphi} f\big) (v)$, $v\in V$. Boundedness of $A$ yields $\dz{\mphi}=\ell^2(V)$ and $A=\mphi$, which completes the proof.
\end{proof}
Combining Corollary \ref{cesaro} and Proposition \ref{sot} we infer that the set of multiplication operators associated to a weighted shift is in fact the closure of the set of polynomials of the weighted shift in question both in the strong and weak operator topology.
\begin{thm}\label{wielomian}
Suppose $\tcal=(V,E)$ is a countably infinite rooted and leafless directed tree, $\lambdab=\{\lambda_v\}_{v \in V^\circ}\subseteq  (0,\infty)$, and $\slam\in\bsb\big(\ell^2(V)\big)$. Then 
\begin{align*}
\multio = \overline{\{ p(\slam) \colon p\in\C[X]\}}^{\text{SOT}} = \overline{\{ p(\slam) \colon p\in\C[X]\}}^{\text{WOT}}.
\end{align*}
\end{thm}
\begin{proof}
Applying Corollary \ref{cesaro}, together with \eqref{Dziobak}, we see that $\multio$ lies in the closure of the set $\{p(\slam) \colon p\in\C[X]\}$ in the strong operator topology. Since, by Lemma \ref{sot}, the space $\multio$ is closed in the strong operator topology, the first of the equalities of the claim has to be satisfied. It is well-known that closures of convex sets in the strong and weak operator topologies coincide, hence the second equality of the claim follows.
\end{proof}

\section{Reflexivity}
In this section we prove that a large class of weighted shifts on directed trees (see Theorem \ref{refPB} below) consists of reflexive operators. 

We begin by describing $\Lat \multio$ and $\Alg \Lat \multio$ in terms of finitely supported multipliers.
\begin{prop}\label{cytrynaPB}
Assume that \eqref{stand2} holds and $\slam\in\bsb\big(\ell^2(V)\big)$. Then
\begin{align}\label{incubus1}
\Lat \multio = \Lat \big\{\mphi \colon \hat\varphi\in\multisk \big\}.
\end{align}
Moreover, we have
\begin{align*}
\Alg \Lat \multio=\Big(\big\{\mphi \colon \hat\varphi\in\multisk\big\}_\perp\cap \bsf_1(\hh)\Big)^\perp;
\end{align*}
in other words,  $\Alg \Lat \multio$ consists of all operators $A\in\bsb(\ell^2(V))$ such that for every $f,g \in \ell^2(V)$ the following condition is satisfied:
\begin{align*}%\label{hive}
\bigg(\forall \hat \varphi\in\multisk\ \ \big\langle M_{\hat \varphi} f, g \big\rangle =0\bigg) \quad  \Longrightarrow \quad \is{Af}{g}=0.
\end{align*}
\end{prop}
\begin{proof}
According to \cite[Theorem 4.3]{b-d-p-2015}, $\{p(\slam)\colon p\in\C[X]\}=\big\{\mphi \colon \hat\varphi\in\multisk \big\}$. Since, by Theorem \ref{wielomian}, $\multio$ is a WOT closure of $\big\{p(\slam)\colon p\in\C[X]\big\}$, we deduce \eqref{incubus1}. Using Theorem \ref{wielomian} and  \cite[Theorem 4.3]{b-d-p-2015} again, we get $\multio_\perp=\big\{\mphi \colon \hat\varphi\in\multisk \big\}_\perp$. Moreover, the preannihilators $\multio_\perp$ and $(\Alg\Lat\multio)_\perp$ have the same rank one operators (see \cite[Chapter 8, $\S$56]{con-ot}). Hence, by \eqref{incubus1}, we get
\begin{align}\label{stained1}
\big\{\mphi \colon \hat\varphi\in\multisk \big\}_\perp\cap \bsf_1(\hh)=(\Alg\Lat\multio)_\perp\cap \bsf_1(\hh).
\end{align}
On the other hand, we have (see \cite[Corollary 2.3]{az-1986-mams})
\begin{align}\label{stained2}
\Alg\Lat\multio=\big((\Alg\Lat\multio)_\perp\cap \bsf_1(\hh)\big)^\perp.
\end{align}
Combining \eqref{stained1} and \eqref{stained2} we get the claim.
\end{proof}
Suppose that \eqref{stand2} holds. Set (cf. \cite[p.\ 12]{b-d-p-2015})
\begin{align*}
r_2^\pth(\slam) = \lim_{n\to\infty}\inf \Big\{\big(\lambda_{\koo|v}\big)^{\frac{1}{|v|}}\colon v\in \pth, |v|\geqslant n\Big\}, \quad \pth\in\pp
\end{align*}
Below we show that the adjoint of any operator belonging to $\Alg\Lat \multio$ can be recovered from the data given by multiplication operators acting along the paths, whenever $r_2^\pth(\slam)>0$ for all $\pth\in\pp$. To this end, we need some notation: given a path $\pth=(V_\pth, E_\pth)\in\pp$ and a function $\hat\varphi_\pth\colon \N_0\to\C$ we put $\lambdab_\pth=\{\lambda_v\}_{v\in V_\pth^\circ}$ and define (using \eqref{multi1} and \eqref{multi2}) the multiplication operator
$M_{\hat \varphi_\pth}^{\lambdab_\pth} \colon \ell^2(V_\pth) \to \ell^2(V_\pth)$ relative to $\pth$.
\begin{lem}\label{alglatPB} 
Assume that \eqref{stand2} holds and $\slam\in\bsb\big(\ell^2(V)\big)$. Let $A\in \Alg\Lat \multio$. Then the following conditions are satisfied$:$
\begin{enumerate}
\item[(i)] $\ell^2(V_\pth) \in  \Lat A^*$ for every $\pth \in \pp$,
\item[(ii)] if $\pth\in\pp$ satisfies $r_2^\pth(\slam)>0$, then there exists a unique function $\hat \varphi_\pth \colon \N_0 \to \C$ such that  $A^*|_{\ell^2(V_\pp)} = \big(M_{\hat \varphi_\pth}^{\lambdab_\pth}\big)^{*}$,
%\item[(iii)] if $r_2^\pth(\slam)>0$ for every $\pth \in \pp$, then $A^* \in \multio^*$.
\end{enumerate}
\end{lem}
\begin{proof} 
(i) Let $\pth \in \pp$. Fix $f\in\ell^2(V_\pth)$ and $v\in V\setminus V_\pth$. Then, by \eqref{Dziobak}, we see that $\is{f}{M_{\hat \varphi} e_v}= 0 $ for every $\hat\varphi\in\multib$. This together with Proposition \ref{cytrynaPB} imply that $\is{A^* f}{e_v}=\is{f}{Ae_v}=0$. Thus $A^* f\subseteq \ell^2(V_\pth)$. This proves (i).

(ii) Fix $\pth\in\pp$. First note that $\ell^2(V_\pth) \in \Lat{M_{\hat \varphi}^*}$ for every $\hat\varphi\in \multi$, which follows from (i) applied to $M_{\hat\varphi}$. Also, for every $\hat\varphi\in \multi$, since $P_{\ell^2(V_\pth)}M_{\hat \varphi} f= M^{\lambdab_\pth}_{\hat\varphi}f$,   $f\in\ell^2(V_\pth)$, we deduce that $\hat\varphi\in \mathcal{M}(\lambdab_\pth)$ and $M_{\hat \varphi}^*|_{\ell^2(V_\pth)}=(M^{\lambdab_\pth}_{\hat\varphi})^*$. 

Now, suppose that $f,g\in \ell^2(V_\pth)$ satisfy $\is{(M_{\hat\varphi}^{\lambdab_\pth})^*f}{g}=0$ for every $\hat\varphi\in\multisk$. Then $\is{M_{\hat\varphi}^*f}{g}=0$ for every $\hat\varphi\in\multisk$. Hence, by Proposition \ref{cytrynaPB}, $\is{A^*|_{\ell^2(V_\pth)}f}{g}=\is{A^*f}{g}=0$. Applying Proposition \ref{cytrynaPB} again, we get $A^*|_{\ell^2(V_\pth)}\in \Alg\Lat \big((\mathscr{M}(\lambdab_\pth)^*\big)$. 

Clearly, $\pth$ is isomorphic with $\N_0$ and $P_{\ell^2(V_\pth)} \slam |_{\ell^2(V_\pth)}$ is unitarily equivalent to the classical weighted shift $S_{\mub}$ on $\ell^2(\N_0)$ with weights $\mub=\{\mu_k\}_{k=0}^\infty$ given by $\mu_k=\lambda_{v_{k+1}}$ with $v_k$, $k\in\N_0$, denoting the unique vertex $v\in V_\pth$ such that $|v_k|=k$ (cf. \cite[Proof of Theorem 6.2]{b-d-p-2015}). Since $r_2^\pth(\slam)>0$, we see that $r_2(S_{\mub})=\liminf_{k\to\infty}(\mu_0\cdots\mu_k)^{\frac{1}{k+1}}>0$, which implies that $S_\mub$ is reflexive (see \cite[Proposition 37]{shi}). In particular, there exists a unique $\hat \varphi_\pth\in \mathcal{M}(\lambdab_\pth)$ such that $A^*|_{\ell^2(V_\pth)}=\big(M^{\lambdab_\pth}_{\hat \varphi_{\pth}}\big)^*$ (see the proof of \cite[Proposition 37]{shi}).

%(iii) Let $\pth_1,\pth_2\in\pp$. It is enough to show that $\hat\varphi_{\pth_{1}}=\hat\varphi_{\pth_2}$. Fix $n\in\N_0$. Let $u,v\in V$ be such that  $|u|=|v|=n $ and $u\in\pth_1$, $v\in\pth_2$.   We define $f=\lambda_{\koo|v}e_u-\lambda_{\koo|u}e_v$. Is is easily seen that for every $k\in\N_0$ we have
%\begin{equation*}
%\langle (\slam^*)^k f, e_{\koo} \rangle=0.
%\end{equation*}
%Proposition \ref{cytrynaPB}  implies that 
%\begin{align*}0&=\langle A^* f, e_{\koo} \rangle
%=\lambda_{\koo|v}\big\langle \big(M_{\hat \varphi_{\pth_1}}^{\pth_1}\big)^*e_u,e_{\koo} \big\rangle-\lambda_{\koo|u}\big\langle \big(M_{\hat \varphi_{\pth_2}}^{\pth_2}\big)^*e_v,e_{\koo} \big\rangle\\
%&=\lambda_{\koo|v}\lambda_{\koo|u}{\hat\varphi_{\pth_1}}(n)-\lambda_{\koo|u}\lambda_{\koo|v}{\hat\varphi_{\pth_2}}(n).
%\end{align*}
%Since $n\in\N_0$ was arbitrary and weights are nonzero we get that $\hat\varphi_{\pth_{1}}=\hat\varphi_{\pth_2}$.
\end{proof}
Recovering $\Alg\Lat \multio$ from the path induced multiplication operators enables us to prove reflexivity for a large class of weighted shifts on directed trees. 
\begin{thm}\label{refPB}
Suppose $\tcal=(V,E)$ is a countably infinite rooted and leafless directed tree, $\lambdab=\{\lambda_v\}_{v \in V^\circ}\subseteq  (0,\infty)$, and $\slam\in\bsb\big(\ell^2(V)\big)$. If $r_2^\pth(\slam)>0$ for every $\pth\in\pp$, then $\slam$ is reflexive.
\end{thm}
\begin{proof}
According to Theorem \ref{wielomian} we need to show that $\multio=\Alg \Lat \multio$. Suppose that $A\in \Alg \Lat \slam$. In view of Proposition \ref{cytrynaPB},  for every $\pth\in\pp$, $A^*|_{\ell^2(V_\pp)} = \big(M_{\hat \varphi_\pth}^{\lambdab_\pth}\big)^{*}$ with some function $\hat \varphi_\pth \colon \N_0 \to \C$. Observe that for any $\pth_1,\pth_2\in\pp$ we have $\hat\varphi_{\pth_{1}}=\hat\varphi_{\pth_2}$. Indeed, fix $\pth_1=(V_{\pth_1}, E_{\pth_1}),\pth_2=(V_{\pth_1}, E_{\pth_1})\in\pp$ and take $n\in\N_0$. Then there are unique $u\in V_{\pth_1}$ and $v\in V_{\pth_2}$ such that $|u|=|v|=n$. Let $f\in\ell^2(V)$ be given by $f=\lambda_{\koo|v}e_u-\lambda_{\koo|u}e_v$. By \eqref{dziobak} we have
\begin{align*}
\langle M_{\hat\varphi}^* f, e_{\koo} \rangle=0,\quad \hat\varphi\in\multisk.
\end{align*}
Therefore, \eqref{dziobak} and Proposition \ref{cytrynaPB}  imply that 
\begin{align*}
\lambda_{\koo|v}\lambda_{\koo|u}\Big({\hat\varphi_{\pth_1}}-{\hat\varphi_{\pth_2}}\Big)(n)
&= \lambda_{\koo|v} \is{\big(M_{\hat \varphi_{\pth_1}}^{\lambdab_{\pth_1}}\big)^*e_u}{e_{\koo}}- \lambda_{\koo|u}\is{\big(M_{\hat \varphi_{\pth_2}}^{\lambdab_{\pth_2}}\big)^*e_v}{e_{\koo}} \\
&=\is{A^* f}{e_{\koo}}=0.
\end{align*}
Since $n\in\N_0$ can be arbitrary we get that $\hat\varphi_{\pth_{1}}=\hat\varphi_{\pth_2}$. As a consequence, there is $\hat\varphi\colon \N_0\to \C$ such that $M^{\lambdab_\pth}_{\hat\varphi_\pth} =M^{\lambdab_\pth}_{\hat\varphi}$ for every $\pth\in\pp$. Since $M_{\hat \varphi}^*|_{\ell^2(V_\pth)}=(M^{\lambdab_\pth}_{\hat\varphi})^*$ (cf. the proof of Lemma \ref{alglatPB}\,(ii)), we get $A^*|_{\ell^2(\pth)}=M_{\hat \varphi}^*|_{\ell^2(V_\pth)}$ for every $\pth\in\pp$ which implies that $A=M_{\hat\varphi}$.  Hence, $\Alg\Lat\multio\subseteq \multio$. Since the reverse inclusion is always satisfied, we get the desired equality and complete the proof.
\end{proof}
\section{Two counterexamples}
In this section we provide two examples related to problems stated in \cite{b-d-p-2015}. The first one (see \cite[Problem 4.9]{b-d-p-2015}) asked whether there exist a weighted shift $\slam$ on a directed tree and a mapping $\hat\varphi\colon\N_0\to \C$ such that $\hat{\varphi} \notin \multib$ and the function $x\mapsto \sum_{k=0}^\infty\hat\varphi(k)z^k$ is bounded on $\{z\in \C\colon |z|<r(\slam)\}$, where $r(\slam)$ is the spectral radius of $\slam$. The Example \ref{mad} below shows that the answer to this problem is affirmative. It is worth noting that the weighted shift $\slam$ constructed in the example is in fact a classical weighted shift on $\ell^2(\N_0)$.
\begin{figure}[h] 
\begin{tikzpicture}[scale=0.8, transform shape]
\tikzstyle{every node} = [circle,fill=gray!30]
\node (e0) at (0,0) {$0$};
\node (e1) at (2,0) {$1$};
\node (e2) at (4,0) {$2$};
\node (e3) at (6,0) {$3$};
\node[fill=none] (ek) at (8,0) {\ldots};
 \draw[->] (e0) --(e1) node[pos=0.5,above = 5pt,fill=none] {$1$};
 \draw[->] (e1) --(e2) node[pos=0.5,above = 5pt,fill=none] {$2$};
 \draw[->] (e2) --(e3) node[pos=0.5,above = 5pt,fill=none] {$\tfrac{3}{2}$};
 \draw[->] (e3) --(ek) node[pos=0.5,above = 5pt,fill=none] {$\tfrac{4}{3}$};
\end{tikzpicture}
\caption{\label{zwyklyShift}}%Weighted shift on $\ell^2(\N_0)$
\end{figure}
\begin{exa}\label{mad}
Let $\tcal=(V,E)$ be the directed tree given by (see Figure \ref{zwyklyShift})
\begin{align*}
V=\N_0,\quad E=\{ (n,n+1) \colon n \in \N_0\}.
\end{align*}
Let $\lambdab=\{\lambda_v\}_{v\in V^\circ}=\{\lambda_n\}_{n \in \N}\subseteq(0,\infty)$ be given by
\begin{align*}
\lambda_n = \left\{
\begin{array}{cl}1 & \text{if } n=1,\\
\frac{n}{n-1} & \text{if } n \neq 1.
\end{array}
\right.
\end{align*}
Finally, let $\slam$ be the weighted shift operator on $\tcal$ with weights $\lambdab$. Then $\slam \in \bsb(\ell^2(V))$ by \cite[Proposition 3.1.8]{j-j-s-2012-mams}. Clearly, we have
\begin{align*}
\lambda_1\cdot \ldots\cdot \lambda_n = n,\quad \lambda_{l+1} \cdot \ldots \cdot \lambda_{l+n} = \frac{l+n}{l},\quad n,l\in\N.
\end{align*}
Thus, by \cite[Lemma 6.1.1 and Proposition 3.1.8]{j-j-s-2012-mams}, we obtain $\|\slam^n\|=n+1$ for every $n\in\N$, and consequently, by the Gelfand's formula for the spectral radius, $r(\slam)=1$.

Now, we define a mapping $\hat{\varphi} \colon \N_0 \to \C$ by 
\begin{align*}
    \hat{\varphi}(k) = \left\{
\begin{array}{cl}
0 & \text{if } k=0,\\
k^{-\tfrac{3}{2}} & \text{if } k>0.
\end{array}
\right.
\end{align*} 
It follows that the function $z\mapsto \sum_{k=0}^{\infty} \hat{\varphi}(k) z^k$ is analytic and bounded in $\{z\in\C\colon |z|<1\}$. On the other hand, by \eqref{dziobak} we have
\begin{align*}
\sum_{k=1}^{\infty} \big|(\varGamma_{\hat{\varphi}}e_0)(k)\big|^2 = \sum_{k=1}^{\infty} |k\,\hat{\varphi}(k)|^2 = \sum_{k=1}^{\infty} \tfrac{1}{k}= \infty.
\end{align*}
Thus $\varGamma_{\hat{\varphi}}e_0\notin\ell^2(V)$, which means that $\hat{\varphi} \notin \multib$.
\end{exa}
The second example answers the question, asked in \cite[Problem 4.11]{b-d-p-2015}, whether for every weighted shift $\slam$ on a directed tree and every $\hat\varphi\in\multib$ we have $|\hat\varphi|\in\multib$? The answer given in Example \ref{offspring} below is negative. It is based on an example of Gaier (see \cite{gai}; we follow here the presentation of the example due to Zalcman from \cite[p. 125]{zal}).
\begin{figure}[ht] 
\begin{tikzpicture}[scale=0.8, transform shape]
\tikzstyle{every node} = [circle,fill=gray!30]
\node (e0) at (0,0) {$0$};
\node (e1) at (2,0) {$1$};
\node (e2) at (4,0) {$2$};
\node (e3) at (6,0) {$3$};
\node[fill=none] (ek) at (8,0) {\ldots};
 \draw[->] (e0) --(e1) node[pos=0.5,above = 5pt,fill=none] {$1$};
 \draw[->] (e1) --(e2) node[pos=0.5,above = 5pt,fill=none] {$1$};
 \draw[->] (e2) --(e3) node[pos=0.5,above = 5pt,fill=none] {$1$};
 \draw[->] (e3) --(ek) node[pos=0.5,above = 5pt,fill=none] {$1$};
\end{tikzpicture}
\caption{\label{zwyklyShift2}}
\end{figure}
\begin{exa}\label{offspring}
By Gaier's example, there exist a Jordan region $G$ and a conformal mapping $f\colon \Dbb\to G$ such that $f$ can be extended to a homeomorphism $F\colon \overline{\Dbb}\to\overline{G}$ with $F(1)=1$ and the length of the image of $[0,1]$ under $F$ is infinite. Let $f(z)=\sum_{n=0}^\infty a_n z^n$ for $z\in\Dbb$. Then we have (see \cite[p. 350]{gai} or \cite[p. 125]{zal})
\begin{align}\label{dreams}
\sum_{n=0}^\infty |a_n|=\int_0^1 |f^\prime(r)|\D r=\infty.
\end{align}
Define $\hat\varphi\colon \N_0\to\C$ by $\hat\varphi(n)=a_n$, $n\in\N$. Let $\tcal$ be the directed tree as in Example \ref{mad} and let $\slam$ be the shift on $\tcal$ with weights $\lambda_k=1$ for $k\in\N$ (this is in fact the classical unilateral shift on $\ell^2(\N_0)$, see Figure \ref{zwyklyShift2}). Then its spectral radius equals $1$, and $\hat\varphi\in\multib$ by \cite[Proposition 4.6]{b-d-p-2015}. On the other hand, $|\hat\varphi|$ cannot belong to $\multib$.  Indeed, if $|\hat\varphi|\in \multib$, then, by \cite[Corollary, p. 76]{shi}, the function $z\mapsto \sum_{k=0}^\infty |\hat\varphi(k)|\, z^k$ has to be analytic and bounded in $\Dbb$. However, in view of \eqref{dreams}, for every $R>0$ there is $N\in\N$ and $q\in(0,1)$ such that $\sum_{k=0}^N |\hat\varphi(k)|>R$ and $q^N> \frac12$. Then we get
\begin{align*}
\sum_{k=0}^\infty |\hat\varphi(k)| q^k \geqslant \sum_{k=0}^N |\hat\varphi(k)| q^k > \frac{R}{2}.  
\end{align*}
Clearly, this is in contradiction with the boundedness of $z\mapsto \sum_{k=0}^\infty |\hat\varphi(k)|\, z^k$.
\end{exa}
\section{The Wold-type decomposition}
In this section we investigate a Wold-type decomposition for weighted shifts on directed trees. We show that the so-called balanced weighted shifts on directed trees have an orthogonal Wold decomposition.

We begin by recalling the some information concerning the adjoint of a weighted shift on a directed tree taken from \cite[Proposition 3.4.1]{j-j-s-2012-mams} and \cite[Lemma 1.2(iii)]{c-t}.
\begin{prop}\label{wodaPB}
Let $\slam\in\bsb(\ell^2(V))$ be a weighted shift on a directed tree $\tcal = (V,E)$ with weights $\lambdab = \{ \lambda_v\}_{v \in V^\circ}$. Then the following assertions are satisfied$:$
\begin{enumerate}
\item[(i)] $\big(\slam^*f\big)(u) = \sum_{v \in \dzi{u} } \overline{\lambda_v} f(v)$ for every $u\in V$ and every $f \in \ell^2(V)$,
\item[(ii)] $\nul{\slam^*} = \Big\{ f \in \ell^2(V) \colon \sum_{v \in \dzi{u}} \overline{\lambda_v} f(v) =0 \text{ for every } u \in V\Big\}$,
\item[(iii)] $(\slam^{*n}\slam^n f)(u)=\|\slam^n e_u\|^2f(u) \text{ for every } f \in \ell^2(V)$.
\end{enumerate}
\end{prop}
\begin{cor}\label{rest}
Let $\slam\in\bsb(\ell^2(V))$ be a weighted shift on a directed tree $\tcal = (V,E)$ with weights $\lambdab = \{ \lambda_v\}_{v \in V^\circ}$. Then $\big\{\chi_{\dzi{v}}f\colon f\in\nul{\slam^*}\big\}\subseteq\nul{\slam^*}$ for every $v\in V$.
\end{cor}
It was proved in \cite[Lemma 3.4, Proof of Theorem 2.7]{c-t} that the Cauchy dual of $\slam$ is analytic, hence $\slam$ has the Wold-type decomposition (see Theorem \ref{studenciPB}(i) below) due to Shimurin's result (see \cite[Proposition 2.7]{shi}). We prove this fact using elementary properties of weighted shifts on directed trees. We provide, in Theorem \ref{studenciPB} below, a more detailed information concerning the decomposition. It is easy to see that the factors in the decomposition are mutually orthogonal for quasinormal operators. A quasinormal weighted shift on a directed tree has very regular distribution of the weights along the whole directed tree. We show that weighted shifts on directed trees whose Wold-type decomposition is orthogonal are exactly those having the same kind of regular distribution of the weights along the generations of the trees. We call those trees balanced. We precede the result with an auxiliary lemma.
\begin{lem}\label{balans}
Assume that \eqref{stand1} holds and $\slam\in\bsb(\ell^2(V))$. Then the following two conditions are equivalent:
\begin{itemize}
\item[(i)] $\|\slam e_u\| = \|\slam e_v\|$ for every $u, \, v \in V$ such that $|u| = |v|$,
\item[(ii)] $\|\slam^n e_{u}\| = \|\slam^n e_{v}\|$ for every $n\in\N$ and all $u, \, v \in V$ such that $|u| = |v|$,
\end{itemize}
Moreover, if $\slam$ is injective, then any of the above conditions is equivalent to the following:
\begin{itemize}
\item[(iii)] $\|\slam^n e_{u}\| = \|\slam^n e_{v}\|$ for every $n\in\N$ and all $u, v \in V$  such that $\pa{u} = \pa{v}$.
\end{itemize}
\end{lem}
\begin{proof}
Assume (i) holds. Then (ii) is satisfied with $n=1$. Assume that it holds for some $n\in\N$. Then for $u,v\in V^\circ$ such that $|u|=|v|$, by \cite[Proposition 3.1.3]{j-j-s-2012-mams}, we have
\begin{align*}
\big\|\slam^{n+1}e_{\pa{u}}\big\|^2&=\big\|\slam^{n}\slam e_{\pa{u}}\big\|^2=\Big\|\slam^{n}\sum_{w\in\dzi{\pa{u}}}\lambda_w e_w\Big\|^2\\
&=\big\| \slam^{n} e_u\big\|^2 \sum_{w\in\dzi{\pa{u}}} |\lambda_w|^2 =\big\| \slam^{n} e_v\big\|^2 \sum_{w\in\dzi{\pa{v}}} |\lambda_w|^2\\
&=\big\|\slam^{n+1}e_{\pa{v}}\big\|^2.
\end{align*}
This, \cite[Proposition 3.1.3]{j-j-s-2012-mams}, and induction yields (ii) for every $n\in\N$.

Evidently, (ii) implies (i) and (iii).

Now, consider the ``moreover'' part of the claim. Assume that (iii) is satisfied. To prove (i) it is enough to show by induction on $k$ that
%\begin{align} \label{indukcja}
%\begin{minipage}{75ex}
%for every $k\in\N$ the equality $\|\slam^n e_u\|= \| \slam^n e_v \|$ holds for every $n \in \N$ and every $u, v \in V$ such that  $\pan{l}{u} = \pan{l}{v}$ with some $l\in \{1,\ldots,k\}$. 
%\end{minipage}
%\end{align} 
\begin{align} \label{indukcja}
\begin{minipage}{75ex}
for every $k\in\N$ the equality $\|\slam^n e_u\|= \| \slam^n e_v \|$ holds for every $n \in \N$ and all $u, v \in V$ such that  $\pan{k}{u} = \pan{k}{v}$. 
\end{minipage}
\end{align} 
For $k=1$ the condition in \eqref{indukcja} is equivalent to (iii). Suppose that the condition in \eqref{indukcja} is valid for some $k \in \N$. Let $u_1$, $u_2 \in V$ be such that $\pan{k+1}{u_1} = \pan{k+1}{u_2}$ and define $v_1 = \pa{u_1}$, $v_2 = \pa{u_2}$. Since $\pan{k}{v_1} = \pan{k}{v_2}$ we have
\begin{equation*}
 \|\slam^{n} e_{v_j}\|^2 = \sum_{w\in\dzi{v_j}}|\lambda_w|^2 \|\slam^{n-1} e_{w}\|^2 = \|\slam^{n-1} e_{u_j}\|^2 \cdot \| \slam e_{v_j}\|^2, \quad j =1,2.
\end{equation*}
By the injectivity of $\slam$ and the induction assumption we have $\|\slam^n e_{v_1}\| = \|\slam^n e_{v_2}\| \neq 0$ for every $n \in \N$. This implies that $\|\slam^{n-1}e_{u_1}\| = \|\slam^{n-1}e_{u_2}\|$ for every $n \in \N$. The induction  gives \eqref{indukcja}, which completes the proof.
\end{proof}
\newpage 
A weighted shift $\slam\in \bsb(\ell^2(V))$ satisfying condition (i) (resp., condition (iii)) of Lemma \ref{balans} is called \emph{balanced} (resp., {\em locally power balanced}).
\begin{thm} \label{studenciPB}
Suppose $\tcal=(V,E)$ is a countably infinite rooted directed tree, $\lambdab=\{\lambda_v\}_{v \in V^\circ} \subseteq  (0,\infty)$, and $\slam\in\bsb(\ell^2(V))$. Then the following three assertions hold:
\begin{itemize}
\item[(i)] $\ell^2(V) = \nul{\slam^*} \oplus \bigvee_{n=1}^{\infty} \slam^n( \nul{\slam^*})$,
\item[(ii)] if $\slam$ is injective, then $\slam^n \big(\nul{ \slam^{*}} \big) \cap \slam^m \big( \nul{ \slam^*} \big)=\{0\}$ for every $n, m \in \N_0$ such that $n\neq m$,
\item[(iii)] if $\slam$ is locally power balanced, then $\slam^n \big(\nul{ \slam^{*}} \big) \perp \slam^m \big( \nul{ \slam^*} \big)$ for every $n, m \in \N_0$ such that $n\neq m$,
\item[(iv)] if $\slam^n \big(\nul{ \slam^{*}} \big) \perp \slam^m \big( \nul{ \slam^*} \big)$ for every $n, m \in \N_0$ such that $n\neq m$, and $\slam$ is injective, then $\slam$ is balanced.
\end{itemize}
\end{thm}
\begin{proof}
(i) Given $n\in\N_0$, we define the subspace $G_n$ of $\ell^2(V)$ by
\begin{align*}
G_n = \ell^2\big(\{|v|=n\}\big).  
\end{align*}
Clearly, the underlying space $\ell^2(V)$ can be decomposed into the orthogonal sum $\bigoplus_{n \in \N_0} G_n$. At the same time we have $\ell^2(V)=\nul{\slam^*}\oplus\overline{\obr{\slam}}$. Thus the decomposition will be deduced by investigating $\nul{\slam^*} \cap G_n$ and $\overline{\obr{\slam}} \cap G_n$ for all $n\in\N_0$.

First, we show that
\begin{align}\label{alt1}
\nul{\slam^*}=\bigoplus_{n \in \N_0} (G_n\cap\nul{\slam^*}).
\end{align} 
Indeed, for every $u,v\in V$ such that $u\neq v$ we have $\dzi{u} \cap \dzi{v} = \varnothing$ and consequently the spaces $\ell^2(\dzi{u})$ and $\ell^2(\dzi{v})$ are orthogonal. Moreover, we have $V = \bigsqcup_{n \in \N_0} \bigsqcup_{|v|=n} \dzi{v}$, where $\bigsqcup$ denotes disjoint union. Hence, by Proposition \ref{wodaPB} (ii), we get
\begin{align}\label{biedronka}
\nul{\slam^*} = G_0 \oplus \bigoplus_{n \in \N_0} \bigoplus_{|u|=n} \Big\{ f \in \ell^2(\dzi{u}) \colon \sum_{v \in \dzi{u}} f(v) \overline{\lambda_v}=0 \Big\}.
\end{align}
This implies \eqref{alt1}.

Now, it follows from \eqref{alt1} that $P_{G_n}P_{\nul{\slam^*}}\ell^2(V)\subseteq P_{\nul{\slam^*}}\ell^2(V)$ for every $n\in\N_0$. 
Hence, by Lemma \ref{prosty} we have
\begin{align}\label{alt2}
P_{\nul{\slam^*}} P_{G_n} = P_{G_n}P_{\nul{\slam^*}},\quad n \in \N_0.
\end{align}

Next, since $\overline{\obr{\slam} \cap G_n}=\big((\obr{\slam} \cap G_n)^\perp\big)^\perp$  and by the continuity of the inner product, $(\obr{\slam}\cap G_n)^\perp=(\overline{\obr{\slam}}\cap G_n)^\perp$  thus
\begin{align}\label{alt3}
\overline{\obr{\slam} \cap G_n} = \overline{\obr{\slam}} \cap G_n\quad \text{for}\quad n \in \N_0.
\end{align}
By Lemma \ref{prosty}, \eqref{alt2}, and \eqref{alt3} we have 
\begin{align}\label{preindy}
G_n=(\nul{\slam^*}\cap G_n)\oplus \overline{\obr{\slam}\cap G_n},\quad n\in\N_0.
\end{align}

Now we prove by the induction that
%This means that proving \eqref{alt4} amounts to showing that
\begin{align}\label{indy}
\overline{\obr{\slam}\cap G_n}=\bigvee_{k=1}^{n} \slam^{k} \Big(\nul{\slam^*} \cap G_{n-k}\Big),\quad n\in\N.
\end{align} For $n=1$, the equality \eqref{indy} is easily seen, since $\overline{\obr{\slam}\cap G_1}=\C \slam(e_{\koo})$ and $\nul{\slam^*} \cap G_{0}=\C e_{\koo}$. Assume that \eqref{indy} holds for some $n\in\N$. Then, in view of \eqref{preindy}, we have

\begin{align*}
\overline{\obr{\slam}\cap G_{n+1}}=\overline{\slam(G_n)}
&=\overline{\slam\Big((\nul{\slam^*}\cap G_n)\oplus \overline{\obr{\slam}\cap G_n}\Big)}\\
&=\overline{\slam\Big((\nul{\slam^*}\cap G_n)\oplus \bigvee_{k=1}^{n} \slam^{k} \big(\nul{\slam^*} \cap G_{n-k}\big)\Big)}\\
&=\overline{\slam\Big(\bigvee_{k=0}^{n} \slam^{k} \big(\nul{\slam^*} \cap G_{n-k}\big)\Big)}\\
&=\bigvee_{k=1}^{n+1} \slam^{k} \big(\nul{\slam^*} \cap G_{n+1-k}\big),
\end{align*}
which yields \eqref{indy}. Equalities \eqref{preindy} and \eqref{indy} imply that
\begin{align}\label{alt4}
G_n = \Big(\nul{\slam^*} \cap G_n\Big) \oplus \bigvee_{k=1}^{n} \slam^{k} \Big(\nul{\slam^*} \cap G_{n-k}\Big),\quad n \in \N.
\end{align}
The combination of \eqref{alt4} with $\ell^2(V)=\bigoplus_{n \in \N_0} G_n$ proves (i).

(ii) Suppose that $\slam$ is injective. Consider $g, h \in \nul{\slam^*}$ such that $\slam^n g=\slam^m h$ for some $n, m\in\N$ such that $n<m$ (the case $m=0$ or $n=0$ follows immediately from (i)). Then, by Proposition \ref{wodaPB} (iii), we get
\begin{align*}
\|\slam^n e_u\|^2g(u) =(\slam^{*n}\slam^{n}g)(u) = (\slam^{*n}\slam^{n}\slam^{m-n}h)(u)=\|\slam^n e_u\|^2(\slam^{m-n}h)(u),\quad u\in V.
\end{align*}
Thus $g=\slam^{m-n}h$. This implies that $g\in \nul{\slam^*}\cap\obr{\slam}=\{0\}$, which gives (ii).

(iii) Assume that $\slam$ is locally power balanced. Fix $n, m \in \N_0$ and $f, g \in \nul{\slam^*}$. We may assume that $n<m$. In view of \eqref{biedronka} and Corollary \ref{rest}, we may also assume that $f=\chi_{\dzi{w}}f$ for some $w\in V$ or $f= e_{\koo}$. We will consider the first case (the second is essentially the same) for fixed $w \in V$. Let $C:=\|\slam^n e_u\|^2$, where $u\in \chi_{\dzi{w}}$. Since $\slam$ is locally power balanced, the definition of $C$ is independent of the choice of $u\in \chi_{\dzi{w}}$. Then, by Proposition \ref{wodaPB} (iii), we get
\begin{align*}\allowdisplaybreaks
\is{\slam^m g}{\slam^n f}&=\is{\slam^{*n}\slam^m g}{f} 
= \sum_{u \in V} \| \slam^n e_u\|^2 \big( \slam^{m-n} g\big) (u) \overline{f(u)}\\
&= \sum_{u\in\dzi{w}} \| \slam^n e_u\|^2 \big(\slam^{m-n} g\big)(u) \overline{f(u)}= C \sum_{u\in\dzi{w}}  \big(\slam^{m-n} g\big)(u) \overline{f(u)}\\
&= C \ \is{\slam^{m-n} g}{\chi_{\dzi{w}} f}=C \ \is{\slam^{m-n} g}{ f}=0,
\end{align*}
where the last equality follows from $f \in \nul{\slam^*}$.

(iv) Suppose, that $\slam^n \big(\nul{ \slam^{*}} \big) \perp \slam^m \big( \nul{ \slam^*} \big)=\{0\}$ for every $n, m \in \N_0$ such that $n\neq m$. Proposition \ref{wodaPB} (iii) implies that
\begin{align}\label{rooster}
0=\is{\slam^{m+n}e_{\koo}}{\slam^n f}=\sum_{u \in V} \|\slam^n e_{u}\|^2 \big(\slam^m e_{\koo}\big)(u) \overline{f(u)},\quad m, n\in\N,\ f\in\nul{\slam^*}.
\end{align}
Let $u_1, u_2\in V$ be such that $\pa{u_1}=\pa{u_2}$. Let $f=\bar\lambda_{u_2}e_{u_1}-\bar\lambda_{u_1}e_{u_2}$ and let $m=|u_1|=|u_2|$. By \eqref{rooster} we get
\begin{align*}
0&=\|\slam^n e_{u_1}\|^2 \big(\slam^m e_{\koo}\big)(u_1) \lambda_{u_2} -\|\slam^n e_{u_2}\|^2 \big(\slam^m e_{\koo}\big)(u_2) \lambda_{u_1}\\
&= \|\slam^n e_{u_1}\|^2 \lambda_{\koo|u_1} \lambda_{u_2} -\|\slam^n e_{u_2}\|^2 \lambda_{\koo|u_2} \lambda_{u_1},\quad n\in \N,
\end{align*}
which yields $\|\slam^n e_{u_1}\|=\|\slam^n e_{u_2}\|$ for every $n\in\N$. This means that $\slam$ is locally power balanced. According to Lemma \ref{balans}, $\slam$ is balanced.
\end{proof}
Using \cite[Proposition 8.1.7]{j-j-s-2012-mams} we get the following.
\begin{cor}
Assume that \eqref{stand1} holds and $\slam\in\bsb(\ell^2(V))$. If $\slam$ is injective and quasinormal, then $\slam^n \big(\nul{ \slam^{*}} \big) \perp \slam^m \big( \nul{ \slam^*} \big)$ for every $n, m \in \N_0$ such that $n\neq m$.
\end{cor}
Regarding Theorem \ref{studenciPB}(ii), the following two elementary examples (see Examples \ref{miotla} and \ref{sscx} below) show that in the non-injective case relations between spaces $\slam^k\big(\nul{\slam^*}\big)$, $k\in\N$, are more complicated. The first one proves also that injectivity of the weighted shift $\slam$ is not necessary to get the condition $\slam^k \big(\nul{\slam^*}\big)\cap \slam^m \big(\nul{\slam^*}\big)=\{0\}$ for all $k,m\in\N$ such that $k\neq m$.
\begin{figure}[ht]
\begin{tikzpicture}[scale=0.8, transform shape,edge from parent/.style={draw,-latex}]
%\tikzstyle{every node} = [circle,fill=gray!30]
\node[circle,fill=gray!30] {$0$}
child {node[circle,fill=gray!30] {$1$} edge from parent node[left,font=\scriptsize]{$\lambda_1 \ \, $}}
child {node[circle,fill=gray!30] {$2$}edge from parent node[left,font=\scriptsize]{$\lambda_2  $}} 
child {node[circle,fill=gray!30] {$3$}edge from parent node[left,font=\scriptsize]{$\lambda_3  $}} 
child {node[circle,fill=gray!30] {$4$}edge from parent node[left,font=\scriptsize]{$\lambda_4  $}} 
child {node[fill = none] {$\ldots \ldots$} edge from parent[->, dashed]  };
\end{tikzpicture}
\caption{\label{miotla-fig}}%Directed tree from Example \ref{miotla}
\label{lisciastyShift}
\end{figure}

\begin{exa}\label{miotla}
Let $\tcal=(V,E)$ be the directed tree given by (see Figure \ref{miotla-fig})
\begin{align*}
V=\N_0,\quad E=\{(0,n)\colon n\in \N\}.
\end{align*}
Let $\lambdab=\{\lambda_{v}\}_{v\in V^\circ}=\{\lambda_{n}\}_{n=1}^\infty$ be non-zero complex numbers such that $\sum_{n=1}^\infty|\lambda_{n}|^2<\infty$. Then $\slam$, the weighted shift on $\tcal$ with weights $\lambdab$ is bounded on $\ell^2(V)$ (see \cite[Proposition 3.1.8]{j-j-s-2012-mams}). Moreover, since $\tcal$ has leafs, $\slam$ is not injective (see \cite[Proposition 3.1.7]{j-j-s-2012-mams}). In view of Proposition \ref{wodaPB} (ii), we have
\begin{align*}
\nul{\slam^*}=\big\{f\in\ell^2(\N_0)\colon  \sum_{n=1}^\infty \bar\lambda_{n}f(n)=0\big\}.
\end{align*}
Then $\slam \big(\nul{\slam^*}\big) = \C \slam e_0$ and $\slam^k \big(\nul{\slam^*}\big) =\{0\}$ for $k\in\N$ such that $k>1$. In particular, we have $\slam^k \big(\nul{\slam^*}\big)\cap \slam^m \big(\nul{\slam^*}\big)=\{0\}$ for all $k,m\in\N$ such that $k\neq m$.
\end{exa}
\begin{figure}[ht] 
\begin{tikzpicture}[scale=0.8, transform shape,edge from parent/.style={draw,-latex}]
%\tikzstyle{every node} = [circle,fill=gray!30]
\node[circle,fill=gray!30] {$0$}
child {node[circle,fill=gray!30] {$1$}child {node[circle,fill=gray!30]{$\omega$} edge from parent node[left,font=\scriptsize]{$\lambda_\omega  $}} edge from parent node[left,font=\scriptsize]{$\lambda_1 \ \, $}  }
child {node[circle,fill=gray!30] {$2$}edge from parent node[left,font=\scriptsize]{$\lambda_2  $}} 
child {node[circle,fill=gray!30] {$3$}edge from parent node[left,font=\scriptsize]{$\lambda_3  $}} 
child {node[circle,fill=gray!30] {$4$}edge from parent node[left,font=\scriptsize]{$\lambda_4  $}} 
child {node[fill = none] {$\ldots \ldots$} edge from parent[->, dashed]  };
\end{tikzpicture}
\caption{\label{sscx-fig}}%Directed tree from Example \ref{sscx}
\label{dwulisciastyShift}
\end{figure}
\begin{exa}\label{sscx}
Let $\tcal=(V,E)$ be the directed tree given by (see Figure \ref{sscx-fig})
\begin{align*}
V=\N_0\cup\{\omega\},\quad E=\{(0,n)\colon n\in \N\}\cup\{(1,\omega)\},
\end{align*}
with $\omega\notin\N_0$. Let $\lambdab=\{\lambda_{v}\}_{v\in V^\circ}=\{\lambda_{n}\}_{n=1}^\infty\cup\{\lambda_{\omega}\}$ be non-zero complex numbers such that $\sum_{n=1}^\infty|\lambda_{n}|^2<\infty$. Clearly, $\slam$, the weighted shift on $\tcal$ with weights $\lambdab$ is bounded on $\ell^2(V)$ and non-injective. In view of Proposition \ref{wodaPB} (ii) we have
\begin{align*}
\nul{\slam^*}=\big\{f\in\ell^2(V)\colon \sum_{n=1}^\infty \bar\lambda_{n}f(n)=f(\omega)=0\big\}.
\end{align*}
Easy calculations show that $\slam \big(\nul{\slam^*}\big) = \bigvee \{ \slam e_0, e_{\omega}\}$ and $\slam^2 \big(\nul{\slam^*}\big) =\C e_{\omega}$, which means that we have $\slam \big(\nul{\slam^*}\big)\cap \slam^2 \big(\nul{\slam^*}\big)\neq\{0\}$. On the other hand, $\slam^k \big(\nul{\slam^*}\big) =\{0\}$ for every $k\in\N$ such that $k>2$.
\end{exa}
In case of a general injective weighted shift $\slam$ on a directed tree, the subspace  $\bigvee_{n=1}^{\infty} \slam^n( \nul{\slam^*})$ hardly decomposes into the orthogonal sum of the factors. In fact, as shown in the example below, representing $f\in\ell^2(V)$ as a $\sum_{n=1}^\infty \slam^n f_n$, $\{f_n\}_{n=1}^\infty\subseteq \nul{\slam^*}$, may not even be possible.
\begin{figure}[ht]
\begin{center}
\begin{tikzpicture}[scale=0.8, transform shape,edge from parent/.style={draw,to}]
\tikzstyle{every node} = [circle,fill=gray!30]
\node (e10)[font=\footnotesize, inner sep = 1pt] at (0,0) {$(0,0)$};

\node (e11)[font=\footnotesize, inner sep = 1pt] at (3,1) {$(1,1)$};
\node (e12)[font=\footnotesize, inner sep = 1pt] at (6,1) {$(1,2)$};
\node (e13)[font=\footnotesize, inner sep = 1pt] at (9,1) {$(1,3)$};
\node[fill = none] (e1n) at(12,1) {};

\node (f11)[font=\footnotesize, inner sep = 1pt] at (3,-1) {$(2,1)$};
\node (f12)[font=\footnotesize, inner sep = 1pt] at (6,-1) {$(2,2)$};
\node (f13)[font=\footnotesize, inner sep = 1pt] at (9,-1) {$(2,3)$};
\node[fill = none] (f1n) at(12,-1) {};

\draw[->=stealth] (e10) --(e11) node[pos=0.5,above = 0pt,fill=none] {$1$};
\draw[->] (e11) --(e12) node[pos=0.5,above = 0pt,fill=none] {$1$};
\draw[->] (e12) --(e13) node[pos=0.5,above = 0pt,fill=none] {$1$};
\draw[dashed, ->] (e13)--(e1n);
\draw[->] (e10) --(f11) node[pos=0.5,below = 0pt,fill=none] {$\alpha$};
\draw[->] (f11) --(f12) node[pos=0.5,below = 0pt,fill=none] {$\alpha$};
\draw[->] (f12) --(f13) node[pos=0.5,below = 0pt,fill=none] {$\alpha$};
\draw[dashed, ->] (f13)--(f1n);
\end{tikzpicture}
\end{center}
\caption{\label{truffaz-fig}}
\label{drzewot201}
\end{figure}
\begin{exa}\label{truffaz}
Let $\tcal_{2} = (V_{2},{E_{2}})$ be the directed tree given by (see Figure \ref{truffaz-fig})\allowdisplaybreaks
\begin{align*}
    V_{2} &= \big\{(0,0) \big\} \cup \big\{ (i,j) \colon  i\in\{1,2\},\ j \in \N \big\}, \\
    E_{2} &= \Big\{ \big((0,0),(i,1)\big) \colon i\in\{1,2\} \Big\}\cup \Big\{ \big((i,j),(i,j+1)\big) \colon i\in\{1,2\},\ j \in \N \Big\}.
\end{align*}
(This directed tree was denoted in \cite{j-j-s-2012-mams} as $\tcal_{2,0}$.) Let $\alpha \in (0,1)$. Let $\slam$ be a weighted shift on $\tcal_{2}$ with weights $\lambdab = \{ \lambda_v\}_{v \in V_{2}}$ defined as follows
\begin{align*}
\lambda_{(i,j)} 
= \left\{ 
\begin{array}{cl} 
1 & \text{ for } i=1 \text{ and } j\in\N, \\ 
\alpha & \text{ for } i=2 \text{ and } j\in\N.
\end{array} 
\right.
\end{align*}
Then $\slam \in \bsb(\ell^2(V))$ by \cite[Proposition 3.1.8]{j-j-s-2012-mams}. In view of Proposition \ref{wodaPB} (ii), we have
\begin{align*}
\nul{\slam^*} = \bigvee\{ e_{00},\alpha e_{11} - e_{21}\}.    
\end{align*}
Note also that
\begin{align}\label{dumb}
\slam^k(e_{00}) = e_{1k} + \alpha^k e_{2k}, \quad \slam^k(\alpha e_{11} - e_{21}) = \alpha e_{1,k+1} - \alpha^k e_{2,k+1},\quad k \in \N.
\end{align}
Now, let $f\colon V_2\to\C$ be given by
\begin{align*}
f((i,j))=
\left\{ 
\begin{array}{cl} 
0 & \text{ for }  i\in\{0,1\} , \\ 
\frac{1}{j} & \text{ for } i=2 \text{ and } j\in\N.
\end{array} 
\right.
\end{align*}
Clearly, $f\in \ell^2(V)$. Suppose that $f$ can be be decompose as $f = \sum_{k=0}^\infty (\beta_k \slam^k e_{00} + \gamma_k \slam^k(\alpha e_{11} - e_{21}))$ for some $\{\beta_k\}_{k \in \N_0}$ and $\{\gamma_k\}_{k \in \N_0}\subseteq \C$. If so, then $f=\lim_{n\to\infty} f_n$, where $f_n=\sum_{k=0}^n \beta_k \slam^k e_{00} + \gamma_k \slam^k(\alpha e_{11} - e_{21})$, $n\in\N$. Employing \eqref{dumb} we get 
\begin{align*}
f_n= \beta_0e_{00} + \gamma_0(\alpha e_{11} - e_{21}) + \sum_{k=1}^n \beta_k (e_{1k} + \alpha^k e_{2k}) + \gamma_k(\alpha e_{1,k+1} - \alpha^k e_{2,k+1}), \quad n\in\N.
\end{align*}
This means that for every $n\in\N$ we have
\begin{align*}
f_n\big((0,0)\big)&= \beta_0,\\
f_n\big((1,j)\big)&= \beta_j + \gamma_{j-1} \alpha, \quad j = 1, \ldots ,n,\\
f_n\big((2,j)\big)&= \beta_j \alpha^j- \gamma_{j-1} \alpha^{j-1}, \quad j = 1, \ldots ,n,\\
f_n\big((1,n+1)\big)&= \gamma_{n} \alpha,\\
f_n\big((2,n+1)\big)&=- \gamma_{n} \alpha^{n}.
\end{align*}
Since $\lim_{n\to\infty}f_n=f$ implies that $\lim_{n\to\infty}f_n(v)=f(v)$ for every $v\in V_2$, we deduce that
\begin{align*}
\gamma_{j} = -\frac{1}{(j+1)\alpha^{j}(1+\alpha^2)},\quad j\in\N_0.
\end{align*}
Therefore, we have
\begin{align*}
\|f-f_n\|^2&\geqslant \Big|f\big((1,n+1)\big)-f_n\big((1,n+1\big))\Big|^2\\
&=\Big|f_n\big((1,n+1\big))\Big|^2
=\frac{1}{(n+1)\alpha^n(1+\alpha^2)},\quad n\in\N.
\end{align*}
Clearly, this contradicts the fact that $\lim_{n\to\infty}\|f-f_n\|=0$.
\end{exa}
\begin{rem}
We note that the Example \ref{truffaz} can be modified so as to show that in fact for every directed tree $\tcal$ which has at least two infinite paths, there exists a weighted shift $\slam$ on $\tcal$ such that the representation $f=\sum_{n=1}^\infty \slam^n f_n$, $\{f_n\}_{n=1}^\infty\subseteq \nul{\slam^*}$ is not possible for some $f$ in the underlying $\ell^2$-space.
\end{rem}
That a locally power balanced weighted shift may not be balanced if injectivity is not assummed (cf. Corollary \ref{balans} and Theorem \ref{studenciPB}(iii)\&(iv)) is shown in the following elementary example.
\begin{exa}
Let $\tcal=\tcal_2$ (see Example \ref{truffaz}). Let $\lambdab=\{\lambda_v\}_{v\in V^\circ}\subseteq\C$ be given by
\begin{align*}
\lambda_{(i,2)}&=0,\quad i\in \{1,2\},\\
\lambda_{(1,j)}&=1,\quad j\in \N\setminus\{2\},\\
\lambda_{(2,j)}&=2,\quad j\in \N\setminus\{2\}.
\end{align*}
Then $\slam \in \bsb(\ell^2(V))$ by \cite[Proposition 3.1.8]{j-j-s-2012-mams}. The only vertexes $u_1, u_2\in V$ such that $u_1\neq u_2$ and $\pa{u_1}=\pa{u_2}$ are $u_1=(1,1)$ and $u_2=(2,1)$. Since $\slam e_{u_1}=\slam e_{u_2}=0$, we get $\|\slam^n e_{u_1}\|=\|\slam^n e_{u_2}\|=0$ for every $n\in \N$. Hence, $\slam$ is locally balanced. Clearly $\slam$ is not balanced.
\end{exa}
The uniqueness of the representation emerges if boundedness from below is assumed.
\begin{prop}
Let $\slam\in\bsb(\ell^2(V))$ be a weighted shift on a directed tree $\tcal = (V,E)$ with weights $\lambdab = \{\lambda_v\}_{v \in V^\circ}$. Assume that $\slam$ is bounded from below. Then for every $f \in \ell^2(V)$, if $f$ has a representation $f = \sum_{n=0}^{\infty} \slam^n f_n$ with $\{f_n\}_{n=0}^\infty \in \nul{\slam^*}$, it is unique.
\end{prop}
\begin{proof}
We first prove that 
\begin{align}\label{basy}
\text{for every $\{g_n\}_{n=0}^\infty\subseteq\nul{\slam^*}$, the equality $\sum_{n=0}^{\infty} \slam^n g_n =0$ implies $g_0=0$.}
\end{align}
Indeed, by Proposition \ref{wodaPB}, we then have
\begin{align}\label{incubus}
0=\slam^*\bigg(\sum_{n=0}^{\infty} \slam^n g_n\bigg)(u)
=\sum_{n=1}^{\infty} (\slam^* \slam^n g_n)(u)
=\sum_{n=1}^{\infty} \|\slam e_u\|^2 (S^{n-1}g_n)(u),\quad  u \in V.
\end{align}
Since $\slam$ is bounded from below, $\|\slam e_u\|\neq 0$ for every $u\in V$, hence \eqref{incubus} yields 
\begin{align*}
\sum_{n=1}^{\infty} \slam^{n}g_n= \slam\bigg(\sum_{n=1}^{\infty} S^{n-1}g_n\bigg)=0.
\end{align*}
Combining this with $\sum_{n=0}^{\infty} \slam^n g_n =0$ gives $g_0=0$.

Now, let $\{f_n\}_{n=0}^\infty\subseteq\nul{\slam^*}$ be such that $\sum_{n=0}^{\infty} \slam^n f_n =0$. Since $\slam$, being bounded from below, is left invertible, we deduce from \eqref{basy} that $\sum_{n=1}^{\infty} \slam^{n-1} f_n =0$. This, by \eqref{basy}, implies that $f_1=0$. Using induction argument we get $f_n=0$ for every $n\in\N_0$.
\end{proof}

\section*{Acknowledgments}
The second, third, and fourth authors were supported by the Ministry of Science and Higher Education of the Republic of Poland.
\bibliographystyle{amsalpha}

\begin{thebibliography}{99}

\bibitem{az-1986-mams} E. A. Azoff, On finite rank operators and preanihilators, Mem. Amer. Math. Soc. {\bf 64} (1986), no.\ 357.

\bibitem{almp} E. A. Azoff, W. S. Li, M. Mbekhta, M. Ptak, {On consistent operators and reflexivity},
 Integr. Equ. Oper. Theory {\bf 71} (2011), 1--12.

\bibitem{b-d-j-s-2014-aaa} P. Budzy\'{n}ski, P. Dymek, Z. J. Jab{\l}o\'nski, J. Stochel, {Subnormal weighted shifts on directed trees and composition operators in $L^2$-spaces with non-densely defined powers}, Abs. Appl. Anal. {\bf 2014} (2014), Article ID 791817, 6 pages.

\bibitem{b-d-p-2015} P. Budzy\'{n}ski, P. Dymek, M. Ptak, {Analytic structure of weighted shifts on directed trees}, Math. Nachr. (2016) (online), doi:10.1002/mana.201500448.

\bibitem{b-j-j-s-2014-jmaa} P. Budzy\'{n}ski, Z. J. Jab{\l}o\'nski, I. B. Jung, J. Stochel {Subnormal weighted shifts on directed trees whose $n$th powers have trivial domain}, J. Math. Anal. Appl. {\bf 435} (2016), 302-314. 

\bibitem{b-j-j-s-wco} P. Budzy\'{n}ski, Z. J. Jab{\l}o\'nski, I. B. Jung, J. Stochel {Unbounded weighted composition operators in $L^2$-spaces}, arXiv: https://arxiv.org/abs/1310.3542.

\bibitem{con} Conway, John B, A course in functional analysis, Graduate Texts in Mathematics, 96, Springer-Verlag, New York, 1990.

\bibitem{con-ot} Conway, John B, A course in operator theory, Graduate Texts in Mathematics, 21, American Mathematical Society, Providence, Rhode Island, 1999.

\bibitem{c-t} S. Chavan, S. Trivedi, {An analytic model for left-invertible weighted shifts on directed trees}, J. London Math. Soc. {\bf 94} (2016), 253-279.

\bibitem{c-p-t} S. Chavan, D. K. Pradhan, S. Trivedi, {Multishifts on directed Cartesian products of rooted directed trees}, arXiv: https://arxiv.org/abs/1607.03860. 

%\bibitem{duren} P. L.  Duren, Theory of $H^p$ spaces, Pure and Applied Mathematics, Vol. 38 Academic Press, New York-London 1970.

\bibitem{gai} D. Gaier, Schlichte Potenzreihen, die auf $|z|=1$ gleichm{\"a}ssig, aber nicht absolutkonvergieren, Math. Zeitschr. {\bf 57} (1952) 349-350.

%\bibitem{g-1969-pams} R. Gellar, {Operators commuting with a weighted shift}, Proc. Amer. Math. Soc. {\bf 23} (1969), 538-545.

\bibitem{hof} K. Hoffman, Banach Spaces of Analytic Functions, Prentice-Hall Series in Modern Analysis Prentice-Hall, Inc., Englewood Cliffs, N. J. 1962.

\bibitem{j-j-s-2012-mams} Z. J. Jab{\l}o\'nski,  I. B. Jung, J. Stochel, {Weighted shifts on directed trees}, Mem. Amer. Math. Soc. {\bf 216} (2012), no.\ 1017.

\bibitem{j-j-s-2012-jfa}  Z. J. Jab{\l}o\'nski,  I. B. Jung, J. Stochel, {A non-hyponormal operator generating Stieltjes moment sequences}, J. Funct. Anal. {\bf 262} (2012), 3946-3980.

%\bibitem{j-j-s-2013-sm} Z. J. Jab{\l}o\'nski, I. B. Jung, J. Stochel {Operators with absolute continuity properties: an application to quasinormality}, Stud. Math., {\bf 215} (2013), 11-30.

\bibitem{j-j-s-2014-pams} Z. J. Jab{\l}o\'nski, I. B. Jung, J. Stochel, {A hyponormal weighted shift on a directed tree whose square has trivial domain}, Proc. Amer. Math. Soc. {\bf 142} (2014), 3109-3116.

%\bibitem{j-j-s-2014-ieot} Z. J. Jab{\l}o\'nski, I. B. Jung, J. Stochel {Unbounded Quasinormal Operators Revisited}, Integr. Equ. Oper. Theory {\bf 79} (2014), 135-149.

%\bibitem{j-l-1979-jot} N. P. Jewell, A. R. Lubin, {Commuting weighted shifts and analytic function theory in several variables}, J. Oper. Theory {\bf 1} (1979), 207-223.

%\bibitem{j-s-2008-jfa} I. B. Jung, J. Stochel, {Subnormal operators whose adjoints have rich point spectrum}, J. Funct. Anal. {\bf 255} (2008), 1797-1816.

\bibitem{nik}  N. K. Nikolskii, Treatise on the shift operator. Spectral function theory. With an appendix by S. V. Khrushchev and V. V. Peller. Translated from the Russian by Jaak Peetre. Grundlehren der Mathematischen Wissenschaften [Fundamental Principles of Mathematical Sciences], 273. Springer-Verlag, Berlin, 1986.

\bibitem{nik2} N. K. Nikolskii, Operators, Functions, and systems: An easy reading. Volume I: Hardy, Hankel and Toeplitz, Mathematical Surveys and Monograps, Volume 92, Amer. Math. Soc., 2002

\bibitem{p-2016-jmaa} P. Pietrzycki, {The single equality $A^{*n}A^n=(A^*A)^n$ does not imply the quasinormality of weighted shifts on rootless directed trees}, J. Math. Anal. Appl. {\bf 435} (2016), 338-348.

\bibitem{kpmp} K. Piwowarczyk, M. Ptak, {On the hyperreflexivity of power partial isometries}, Linear Algebra Appl. {\bf 437} (2012), 623-629.

\bibitem{sa} D. Sarason, {Invariant subspaces and unstarred operator algebras}, Pacific J. Math. {\bf 17} (1966), 511--517.

\bibitem{shi} A. L. Shields, {Weighted shift operators and analytic function theory}, Topics in operator theory, pp. 49-128. Math Surveys, No. 13, Amer. Math. Soc., Providence, R.I., 1974.

\bibitem{s-2001-jram} S. Shimorin, {Wold-type decompositions and wandering subspaces for operators close to isometries}, J. Reine Angew. Math. {\bf 531} (2001), 147-189

%\bibitem{s-s-1989-rims} J. Stochel, F. H. Szafraniec, {On normal extensions of unbounded operators. III. Spectral properties}, Publ. RIMS, Kyoto Univ. {\bf 25} (1989), 105-139.

%\bibitem{s-2000-otaa} F. H. Szafraniec, {The reproducing kernel Hilbert space and its multiplication operators}, Operator Theory: Advances and Applications {\bf 114} (2000), 253-263

%\bibitem{s-2003-otaa} F. H. Szafraniec, {Multipliers in the reproducing kernel Hilbert space, subnormality and noncommutative complex analysis}, Operator Theory: Advances and Applications {\bf 143} (2003), 313-331.

\bibitem{t-2015-jmaa} J. Trepkowski, {Aluthge transforms of weighted shifts on directed trees}, J. Math. Anal. Appl. {\bf 425} (2015), 886-899.

%\bibitem{x-1983-ieot} D. Xia, {On the analytic model of a class of hyponormal operator}, Integral Eq. Oper. Theory {\bf 6} (1983), 134-157.

\bibitem{zal} L. Zalcman, {Real proofs of complex theorems (and vice versa)}, Amer. Math. Monthly  {\bf 81}  (1974), 115-137. 
        \end{thebibliography}

\end{document}